\documentclass[10pt]{amsart}
\usepackage{amssymb,amsmath,textcomp,amsfonts,amsthm}
\usepackage[utf8]{inputenc}
\usepackage{url}
\usepackage{mathtools}
\usepackage{romannum}
\usepackage{color}
\usepackage[normalem]{ulem}
\usepackage{tikz}
\usepackage{scrextend}
\usepackage[version=4]{mhchem}
\newtheoremstyle{normal}
{10pt}
{10pt}
{\normalfont}
{}
{\bfseries}
{}
{0.8em}
{\bfseries{\thmname{#1}\thmnumber{ #2}.\thmnote{ \hspace{0.5em}(#3)\newline}}}
 
\newtheoremstyle{kursiv}
{10pt}
{10pt}
{\itshape}
{}
{\bfseries}
{}
{0.8em}
{\bfseries{\thmname{#1}\thmnumber{ #2}.\thmnote{ \hspace{0.5em}(#3)\newline}}}

\theoremstyle{kursiv}
\newtheorem{definition}{Definition}[section]
\newtheorem{satz}[definition]{Theorem}
\newtheorem{lemma}[definition]{Lemma}
\newtheorem{korollar}[definition]{Corollary}
\newtheorem{proposition}[definition]{Proposition}

\theoremstyle{normal}

\newtheorem{bemerkung}[definition]{Remark}
\newtheorem{beispiel}[definition]{Example}

\newlength{\leftstackrelawd}
\newlength{\leftstackrelbwd}
\def\leftstackrel#1#2{\settowidth{\leftstackrelawd}%
{${{}^{#1}}$}\settowidth{\leftstackrelbwd}{$#2$}%
\addtolength{\leftstackrelawd}{-\leftstackrelbwd}%
\leavevmode\ifthenelse{\lengthtest{\leftstackrelawd>0pt}}%
{\kern-.5\leftstackrelawd}{}\mathrel{\mathop{#2}\limits^{#1}}}

\newcommand{\iR}{\mathbb{R}}
\newcommand{\R}{\mathbb{R}}
\newcommand{\iN}{\mathbb{N}}
\newcommand{\N}{\mathbb{N}}
\newcommand{\iZ}{\mathbb{Z}}
\newcommand{\Z}{\mathbb{Z}}

\DeclareMathOperator*{\dist}{dist}

\allowdisplaybreaks
\begin{document}
\pagenumbering{arabic}
\title{Li-Yau type and Harnack estimates for systems of reaction-diffusion equations via hybrid
curvature-dimension condition}
\author{Sebastian Kr\"ass}
\email{sebastian.kraess@uni-ulm.de}
\author{Rico Zacher$^*$}
\thanks{$^*$Corresponding author}
\email[Corresponding author:]{rico.zacher@uni-ulm.de}
\address[Sebastian Kr\"ass, Rico Zacher]{Institute of Applied Analysis, Ulm University, Helmholtzstra\ss{}e 18, 89081 Ulm, Germany.}
\thanks{Sebastian Kr\"ass is supported by a PhD-scholarship of the ``Hanns-Seidel-Stiftung'', Germany. }

\begin{abstract}
We prove Li-Yau and Harnack inequalities for systems of linear reaction-diffusion equations.
By introducing an additional discrete spatial variable, the system is rewritten as a scalar
diffusion equation with an operator sum. For such operators in a mixed continuous and discrete
setting, we introduce the hybrid curvature-dimension condition $CD_{hyb} (\kappa,d)$, which is a combination of the Bakry-Émery condition $CD(\kappa,d)$ and one of its discrete analogues, the condition $CD_\Upsilon (\kappa,d)$. We establish a hybrid tensorisation principle and prove that under
$CD_{hyb} (0,d)$ with $d<\infty$ a differential Harnack estimate of Li-Yau type holds, from which a Harnack inequality can 
be deduced by an integration argument.
\end{abstract}
\maketitle
\vspace{0.7em}
\noindent{\bf AMS subject classification (2020):} 35K57 (primary), 60J60 (secondary)

${}$

\noindent{\bf Keywords:} Li-Yau inequality, Harnack inequality, cooperative parabolic system, 
curvature-dimension condition, hybrid setting, tensorisation principle

\section{Introduction}
Harnack inequalities play an important role in the theory of partial differential equations,
geometric analysis and the theory of stochastic processes. They are well-known for large classes of elliptic and parabolic PDEs. In particular, nonnegative (smooth) solutions of the heat equation
\begin{equation} \label{heat1}
\partial_t u(t,x)-\Delta u(t,x)=0,\quad t>0,\;x\in \iR^n,
\end{equation}
satisfy the following Harnack inequality (see e.g.\ \cite{AUB,Moser64})
\begin{equation} \label{Har1}
u(t_1,x_1)\le u(t_2,x_2) \Big(\frac{t_2}{t_1}\Big)^{n/2} e^{\frac{|x_1-x_2|^2}{4(t_2-t_1)}},\quad 0<t_1<t_2,
\,x_1,x_2\in \iR^n,
\end{equation}
which is optimal; in fact, one has equality if one takes for $u$ the fundamental
solution of the heat equation. 

An important approach to derive Harnack estimates such as \eqref{Har1} is based on {\em differential Harnack inequalities} and
goes back to the fundamental contribution made by Li and
Yau (\cite{LY}), who were able to prove the Harnack inequality for Riemannian manifolds
with Ricci curvature bounded from below. In case of the heat equation on $\iR^n$, any positive 
smooth solution $u$ satisfies the Li-Yau inequality
\begin{equation} \label{LiYauklass}
- \Delta \log u \leq \frac{n}{2t}\quad \mbox{in}\;(0,\infty)\times \iR^n,
\end{equation}
which is equivalent to the differential Harnack inequality
\begin{equation} \label{CLY}
 |\nabla \log u|^2-\partial_t\log u\le \frac{n}{2t}\quad \mbox{in}\;(0,\infty)\times \iR^n,
\end{equation}
which in turn, by integration along a suitable path in space-time, yields the classical parabolic Harnack estimate \eqref{Har1}.

The main aim of this paper is to establish Li-Yau and Harnack inequalities for linear systems of reaction-diffusion
equations of the form
\begin{equation} \label{RDSystems}
\partial_t u_i (t,x) - \Delta u_i(t,x) = \sum_{j=1}^m k(i,j)\big(u_j(t,x)-u_i(t,x)\big), 
\quad t>0,\ x\in \iR^n,\ i = 1,\dots,m, 
\end{equation}
where $k$ is a given nonnegative kernel. If we consider $u_1,\ldots,u_m$ as concentrations of $m$ chemical
species $A_1,\,\ldots, A_m$ and assuming that $k$ is symmetric, the term $k(i,j) (u_j-u_i)$ can be interpreted
as the rate of change of the concentration $u_i$ due to the reversible chemical reaction 
$ \ce{A_i <=>A_j}$ with reaction rate $k(i,j)$ (in both directions), provided that $k(i,j)>0$. The simplest example
is given by the system
\begin{equation} \label{OneReacExample}
\left\{
\begin{array}{r@{\;=\;}l@{\;}l}
\partial_t u_1 (t,x) - \Delta u_1 (t,x) & u_2(t,x) -u_1 (t,x),\quad & t>0,\,x\in \iR^n, \\ 
 \partial_t u_2 (t,x) - \Delta u_2 (t,x) & u_1(t,x) -u_2 (t,x),\quad & t>0,\,x\in \iR^n,
\end{array}
\right.
\end{equation}
corresponding to just one reversible reaction $ \ce{A_1 <=>A_2}$ with reaction rate $1$ for both directions.

In general, systems do not allow for a maximum principle, so that one cannot expect to obtain Harnack inequalities, let alone Li-Yau type estimates. In case of \eqref{RDSystems}, the particular structure of the right-hand side, consisting of weighted
sums of differences $u_j-u_i$, suggests that it may be possible to derive pointwise estimates such as Harnack inequalities similar to the scalar case. In fact, \eqref{RDSystems} falls into
the class of {\em cooperative} parabolic systems, for which a Harnack inequality of the form
\begin{equation} \label{HarPol}
|u_i|_{L_\infty((\theta,2\theta)\times K)}\le C\inf_{[\frac{7}{2}\theta,4\theta]\times K} u_j,
\quad i,j\in \{1,\ldots,m\},
\end{equation}
has been established by F\"oldes and Pol\'a\^{c}ik in \cite[Theorem 3.9]{FoePol} under an irreducibility condition on $k$ for sufficiently smooth nonnegative 
solutions $u=(u_1,\ldots,u_m)$ on $(0,4\theta]\times \Omega$, where $\theta>0$, $K$
is a compact subset of the domain $\Omega$, and the constant $C$ depends on
$\theta, n, m$ and $K$, but not on $u$. Denoting the right-hand side of \eqref{RDSystems}
by $f_i(u)$, cooperativity means that $\frac{\partial f_i}{\partial u_j}(u)=k(i,j)\ge 0$ for all 
$i,j\in\{1,\ldots,m\}$ with $i\neq j$. We refer to \cite{FoePol} for more details on this notion
in a more general (nonlinear) setting.

Our approach to proving Li-Yau inequalities is based on the following idea. 
Considering the index $i\in \{1,\ldots,m\}$ as an additional spatial variable, which is
discrete (!), the system \eqref{RDSystems} can be viewed as a {\em scalar diffusion equation} of the form
\begin{equation} \label{Lequation}
 \partial_t u-\mathcal{L}u=0
\end{equation}
for the function
$u(t,x,i)$ defined by $u(t,x,i)=u_i(t,x)$ on $(0,\infty)\times \iR^n\times \{1,\ldots,m\}$, where the operator
$\mathcal{L}$ is the sum of the Laplacian $\Delta$ acting w.r.t.\ the continuous variable $x$ and a discrete
diffusion operator $L_d$ (given by the right-hand side in \eqref{RDSystems}) which acts w.r.t.\ the second
spatial variable $i$. This observation is the starting point for our study. Instead of the system \eqref{RDSystems},
we consider a scalar heat (or diffusion) equation with an operator sum $\mathcal{L}$ acting on a product structure in space, which is {\em hybrid} in the sense that we have a continuous and a discrete spatial variable.

To make its role as a second spatial variable clearer in the notation, we will from now on denote the discrete variable by $y$. Let $Y$ be a finite set (with at least two elements) and $k: Y\times Y\rightarrow [0,\infty)$. 
The operator $L_d$ acts on functions $w: Y\rightarrow \iR$ and is defined by
\begin{equation} \label{LdDef}
L_d w(y)=\sum_{z\in Y} k(y,z)\big(w(z)-w(y)\big),\quad y\in Y.
\end{equation}
Here the '$d$' in the notation refers to 'discrete' and should not be confused with any parameter $d$ later on. Observe that the values of $k(y,y)$ for $y\in Y$ do not play any role. We will always assume that 
$k(y,y)=0$ for all $y\in Y$. $L_d$ can be viewed as the graph Laplacian associated to the directed graph with
vertex set $Y$ where the edges are all pairs $(y,z)\in Y\times Y$ with $k(y,z)>0$ and $k(y,z)$ are the 
corresponding edge weights.

The operator $L_d$ is also the generator of a Markov chain on $Y$, where
$k(y,z)$ is the transition rate for jumping from $y$ to $z$ ($\neq y$). Note that in the theory of
Markov chains one usually writes the generator in the equivalent form $L_d w(y)=\sum_{z\in Y}k(y,z) w(z)$ with the kernel $k$ satisfying $\sum_{z\in Y}k(y,z)=0$ for all $y\in Y$, which
leads to the condition
$k(y,y)=-\sum_{z\in Y\setminus \{y\}}k(y,z)\le 0$, $y\in Y$. 

In order to derive Li-Yau and Harnack inequalities for positive solutions $u$ of the equation
\begin{equation} \label{DLdEquation}
\partial_t u(t,x,y)-\Delta u(t,x,y)-L_d u(t,x,y)=0,\quad t>0,\,x\in \iR^n,\,y\in Y,
\end{equation}
we want to use the approach of Li and Yau \cite{LY}. In the purely continuous setting such as the heat equation with
Laplace-Beltrami operator $\Delta$ on a complete Riemannian manifold $M$, a crucial role is played by the 
$\Gamma$-calculus of Bakry and \'Emery, a powerful theory at the heart of which is the {\em curvature-dimension condition} $CD(\kappa,n)$, see \cite{BGL,BE} and Definition \ref{CDDef} below. This condition takes its name from the fact that the Laplace-Beltrami operator $\Delta$ on a complete Riemannian manifold $M$ satisfies $CD(\kappa,n)$ if and only if the Ricci curvature of the manifold is bounded from below by $\kappa$ and $n$ is greater than or equal to the topological dimension of $M$. In case of the validity of $CD(0,n)$, positive solutions $u$ of the heat equation on $M$ satisfy the Li-Yau inequality \eqref{LiYauklass}
with $M$ instead of $\iR^n$. The generalisation of this result to Markov diffusion operators satisfying $CD(0,n)$ has been elaborated in \cite{BL} and can be also found in the monograph \cite{BGL}, which contains
many more applications of $CD$-conditions such as various types of functional inequalities.

In the discrete setting, various versions of the celebrated Li-Yau inequality have been proven during
the last decade (\cite{Harvard,DKZ,MUN}). The main problem here is the failure of the chain rule for the 
graph Laplacian (in particular when applied to $\log u$) so that one has to find a suitable discrete calculus and an appropriate $CD$-condition which make the Li-Yau approach work.
 
First results in this direction were obtained by Bauer, Horn, Lin, Lippner, Mangoubi and Yau \cite{Harvard}. They circumvent the problem of the failure of the chain rule by considering
the square root instead of the logarithm. Under the so-called {exponential
curvature-dimension inequality} $CDE(\kappa,n)$ they obtain Li-Yau and Harnack
inequalities for a wide class of locally finite graphs including Ricci-flat graphs. A few years later, M\"unch \cite{MUN} introduced a new calculus, called $\Gamma^\psi$-calculus, where $\psi$ is a concave function. Assuming that a finite graph satisfies the so-called $CD\psi(n,0)$-condition, M\"unch derives a Li-Yau estimate which for 
$\psi=\sqrt{}$ coincides with one of those obtained in \cite{Harvard} and which in case $\psi=\log$
yields the more natural logarithmic estimate 
$-L_d \log u \le \frac{n}{2t}$. M\"unch also showed that 
in some examples (e.g.\ Ricci-flat graphs) the choice $\psi=\log$ leads to better estimates than the square-root approach.
Recently, Dier, Kassmann and Zacher \cite{DKZ} proposed another $CD$-condition in the discrete setting, 
the condition $CD(F;0)$, which allows 
to significantly improve the Li-Yau estimates from \cite{Harvard,MUN}. For example, they were able to prove a sharp logarithmic
Li-Yau inequality for the unweighted two-point graph and a Li-Yau estimate for the lattice $\tau \iZ$ with grid size $\tau$
which for $\tau\to 0$ leads to the classical sharp Li-Yau inequality on $\iR$.

In the present paper we will use the condition $CD_\Upsilon (\kappa,d)$ for the discrete
part, see Definition \ref{CDPSIDEF} below. This condition was introduced by Weber and Zacher in \cite[Definition 2.7 and Remark 2.9]{WZ}, who also showed that in case $d=\infty$ it is equivalent to 
M\"unch's condition $CD\log(d,\kappa)$ mentioned in \cite[Remark 3.12]{MUN}; this equivalence
is also true for $d<\infty$, see Remark \ref{BemerkungMuench}.
The function $\Upsilon (x) = e^x -1-x$ plays a key role in this calculus, which is much inspired
by \cite{DKZ}. In fact, $CD_\Upsilon(0,d)$ is closely linked to the condition $CD(F;0)$ with a quadratic $CD$-function $F$. In the special case of $CD_\Upsilon(\kappa,\infty)$, there are further close connections to the generalized $\Gamma$-calculus described in Monmarch\'e \cite[Section 2]{Mon} and the notion of entropic Ricci curvature of Erbar and Maas in \cite{EM12}, see also \cite{Mie}. For more details we refer to \cite[Remark 2.9]{WZ}.

By combining elements of the Bakry-\'Emery condition $CD(\kappa,n)$ for the 
continuous part and those of the condition $CD_\Upsilon (\kappa,d)$ for the discrete
component we introduce a \emph{hybrid curvature-dimension condition} 
$CD_{hyb} (\kappa,d)$ for general operator sums $L_c \oplus L_d$ acting on 
(suitable) functions defined on the product space $X\times Y$, see Definition 
\ref{DefCDCond} below. Here $L_c$ is a Markov diffusion operator acting w.r.t.\ the (continuous) variable $x\in X$ and $L_d$ is as in \eqref{LdDef} and acts w.r.t.\ the discrete variable $y\in Y$. Our definition of $CD_{hyb}(\kappa,d)$ is based on ideas from
\cite[Section 7]{WZ}, where $CD_{hyb}(\kappa,\infty)$ with $\kappa>0$ appears implicitly 
in the proof of the modified logarithmic Sobolev inequality in the hybrid setting. 

The following {\em hybrid tensorisation principle} is of central importance for our analysis,
see Theorem \ref{SatzHinrCD} below.
If $L_c$ satisfies $CD(\kappa_1,n)$ and $L_d$ satisfies $CD_\Upsilon (\kappa_2,d)$ for some $\kappa_1,\kappa_2 \in \R$ and $n,d \in [1,\infty]$, then the operator $L_c \oplus L_d$ satisfies $CD_{hyb} (\kappa,\tilde{d})$ with $\kappa = \min \{\kappa_1,\kappa_2\}$ and $\tilde{d} = n+d$. Since $CD(0,n)$ holds for the Laplacian $L_c=\Delta$ on $\iR^n$, by tensorisation
we directly obtain $CD_{hyb} (0,\tilde{d})$ with some $\tilde{d}\in [n+1,\infty)$ for the sum $\mathcal{L}=\Delta \oplus L_d$ on $\iR^n \times Y$, provided that the discrete operator satisfies 
$CD_\Upsilon (0,d)$ with some $d\in [1,\infty)$. The latter property can be verified for 
various examples, an important class being given by Ricci-flat graphs (see Theorem \ref{SatzRicciflat}).

The condition $CD_{hyb} (0,d)$ with $d\in [1,\infty)$ is tailor-made for establishing
Li-Yau and differential Harnack estimates for positive smooth solutions $u$  to the equation $\partial_t u - \mathcal{L} u= 0$ on $(0,\infty) \times \R^n \times Y$. In fact, as one of
the main results of this paper (see Corollary \ref{KorollarLY}) we obtain that in this situation 
\begin{equation} \label{LiYauGlob}
- \mathcal{L} \log u = \big( \vert \nabla \log u \vert^2 + \Psi_\Upsilon (\log u)\big) - \partial_t \log u \leq \frac{d}{2t}\;\; \text{ on } (0,\infty) \times \R^n \times Y,
\end{equation}
where $\nabla$ acts w.r.t.\ $x$ and $\Psi_\Upsilon (\log u)$ is the discrete counterpart of $ \vert \nabla \log u \vert^2$, see \eqref{PsiUps} for its definition. To prove estimate \eqref{LiYauGlob}, we first show a related  Li-Yau inequality for solutions which are only
local w.r.t.\ the continuous variable and then take a limit by increasing the radius of the
underlying ball to infinity.

With the Li-Yau resp.\ differential Harnack estimates at hand, we are also able to derive
Harnack inequalities for positive smooth solutions $u$  to the equation $\partial_t u - \mathcal{L} u= 0$. Here we allow also for infinite graphs $Y$, but impose additional
conditions on the kernel $k$. We assume that $k(y_1,y_2)>0$ if and only if $k(y_2,y_1)>0$,
for all $y_1,y_2\in Y$. Moreover, it is assumed that there is a $k_{\min}>0$ such that all positive edge weights $k(y_1,y_2)$ are bounded from
below by  $k_{\min}$. Letting $u$ be a positive smooth function on $(0,\infty) \times \R^n \times Y$ and assuming validity of \eqref{LiYauGlob}, we can show that for any $x_1,x_2 \in \R^n$,
any $y_1,y_2\in Y$ and any $0 < t_1 < t_2 < \infty$ there holds
\begin{align}
u(t_1,x_1,y_1) \leq u(t_2,x_2,y_2) \Big(\frac{t_2}{t_1}\Big)^\frac{d}{2} \exp \Big( \frac{\vert x_2 - x_1 \vert^2}{4 (t_2-t_1)} + 2 \frac{\dist (y_1,y_2)^2}{k_{\min} (t_2-t_1)}\Big),
\label{HarGlob}
\end{align} 
see Corollary \ref{KorollarHarnack} below. We also have a corresponding result (Theorem \ref{SatzHarnack}) which applies to positive solutions which are only
local w.r.t.\ the continuous variable. 

In case of the example \eqref{OneReacExample}, to which our theory applies, the inequality \eqref{HarGlob} yields the 
single-species estimates
\begin{align}
u_i(t_1,x_1) \leq u_i(t_2,x_2) \Big(\frac{t_2}{t_1}\Big)^\frac{d}{2} e^{ \frac{\vert x_2 - x_1 \vert^2}{4 (t_2-t_1)}},\quad
i=1,2,
\label{OR1}
\end{align} 
and the two-species estimates
\begin{align}
u_i(t_1,x_1) \leq u_{3-i}(t_2,x_2) \Big(\frac{t_2}{t_1}\Big)^\frac{d}{2} \exp \Big( \frac{\vert x_2 - x_1 \vert^2}{4 (t_2-t_1)} + \frac{2}{t_2-t_1}\Big),\quad i=1,2,
\label{OR2}
\end{align} 
for $0 < t_1 < t_2 < \infty$ and $x_1,x_2 \in \R^n$, where $d\approx n+1.26$ is a constant which can be computed explicitly;
in particular the estimates are valid with $d=n+2$. See also Example \ref{Beispiele}(i). Observe that the single-species estimates \eqref{OR1} take
the same form as \eqref{Har1}, except for the increased exponent $d$. This is a general property, since the second
summand inside the brackets in \eqref{HarGlob} vanishes for $y_1=y_2$.

In this paper, for the sake of simplicity, we have restricted ourselves to the case of the Laplace operator on $\iR^n$ with respect to the continuous diffusion component. However, since the hybrid
tensorisation principle holds for general Markov diffusion operators, our results can be extended
to certain second-order parabolic equations on Riemannian manifolds with 
Ricci curvature bounded from below. In particular, in the case where $L_c$ is the Laplace-Beltrami operator our proofs of the Li-Yau and Harnack estimates can be adapted by using the
arguments from \cite[Chapter 12]{Li}.

This article is organised as follows. In Section 2, after recalling the conditions $CD(\kappa,d)$ and
$CD_\Upsilon(\kappa,d)$, we introduce the hybrid curvature-dimension condition $CD_{hyb} (0,d)$
and prove the hybrid tensorisation principle. We further discuss several examples where
the operator $\mathcal{L}=\Delta\oplus L_d$ satisfies $CD_{hyb} (0,d)$ with some finite $d\ge 1$. 
Section 3 is devoted to the derivation of Li-Yau estimates, which are then used in Section 4 to prove
Harnack inequalities. We also illustrate our results with a few examples of concrete reaction-diffusion systems. In the final Section 5, we discuss some open problems with regard to possible extensions
of our results in various directions.

\section{Curvature-dimension conditions}

In this section, we will introduce a new curvature-dimension condition which plays a crucial role in our proof of the Li-Yau inequality for positive solutions to the diffusion equation \eqref{DLdEquation} in the hybrid continuous-discrete setting. We will call this condition the \emph{hybrid curvature-dimension condition} as it combines already existing conditions from the continuous and the discrete setting. 

The starting point is the well known Bakry-Émery condition $CD(\kappa,n)$, which for nonnegative $\kappa$ (the curvature parameter) was shown to be sufficient for proving the Li-Yau estimate in the continuous setting for Markov diffusion operators, see \cite{BGL}. We recall this condition for the reader's convenience. Let $L_c$ be the generator of a Markov semigroup with a fixed invariant and reversible measure $\mu$
on the underlying state space $X$, e.g.\ a Riemannian manifold. The carr\'e du champ operator $\Gamma$ and the iterated
carr\'e du champ operator $\Gamma_2$ associated with $L_c$ are defined on a suitable algebra of functions as
\begin{align*}
\Gamma (u,v) & = \frac{1}{2} \big(L_c (uv) -u L_c v - v L_c u\big),\\
\Gamma_2 (u,v) & = \frac{1}{2}\big(L_c \Gamma(u,v)-\Gamma(u, L_c v)-\Gamma(L_c u,v)\big).
\end{align*}
We set $\Gamma (u) := \Gamma (u,u)$ and $\Gamma_2(u) := \Gamma_2(u,u)$. The Bakry-Émery condition is now defined as follows (see \cite[Definition 1.16.1]{BGL}).

\begin{definition} \label{CDDef}
The Markov generator $L_c$ is said to satisfy the condition $CD(\kappa,n)$ with $\kappa \in \R$ and $n \in [1,\infty]$, if for every function $f$ in a sufficiently rich class $\mathcal{A}$ there holds
\begin{align*}
\Gamma_2 (f) \geq \kappa \Gamma (f) + \frac{1}{n} (L_c f)^2,\quad \mu - \mbox{a.e.}
\end{align*}
\end{definition}

Motivated by second-order differential operators without zeroth order term, one says that the Markov generator $L_c$ satisfies the {\em diffusion property}
and calls $L_c$ a Markov diffusion operator if for any smooth function $\psi: \R \to \R$ and any suitable function $f$ we have
\begin{equation}
\label{DiffProp}
L_c \psi (f) = \psi^\prime (f) L_c(f) + \psi^{\prime \prime} (f) \Gamma(f),
\end{equation}
see also \cite[Definition 1.11.1]{BGL}. 

\begin{bemerkung}
An important example for a Markov diffusion operator is given by the classical Laplacian $\Delta$ on $\R^n$ with $\mu$ being the Lebesgue measure. Here
\begin{align*}
\Gamma (u,v) = \langle \nabla u, \nabla v \rangle \text{ and } \Gamma_2 (u) = \vert \nabla^2 u \vert_{HS}^2,
\end{align*}
and the diffusion property follows directly from the chain rule. It is well known that $\Delta$ satisfies $CD(0,n)$. Indeed,
\begin{align}
\Gamma_2(u) = \vert \nabla^2 u \vert^2_{HS} = 
\sum_{i=1}^n \sum_{j=1}^n \Big(\partial_{x_i} \partial_{x_j}u \Big)^2 \geq 
\sum_{i=1}^n \Big(\partial^2_{x_i}u  \Big)^2 \geq \frac{1}{n} \Big( \sum_{i=1}^n \partial^2_{x_i}u  \Big)^2 = \frac{1}{n} \big(\Delta u\big)^2.
\label{BakryCD(0,d)}
\end{align}
\end{bemerkung}

As mentioned before, the hybrid curvature-dimension condition also incorporates a discrete curvature-dimension condition, namely the condition $CD_\Upsilon (\kappa,d)$. This condition was introduced in \cite[Definition 2.7 and Remark 2.9]{WZ} in the discrete setting of Markov chains and is motivated by certain representation formulas for the first and second time derivative of the discrete Boltzmann entropy along the heat flow generated by the generator of the chain. Note that $CD_\Upsilon (\kappa,d)$ implies the Bakry-Émery condition $CD(\kappa,d)$; this was proved in \cite[Proposition 2.11]{WZ} only in the case $d=\infty$, but the argument given there easily extends to the general case. The examples in \cite[Section 5]{WZ} further show that $CD_\Upsilon (\kappa,d)$ is stronger
than $CD(\kappa,d)$. 

To describe the $CD_\Upsilon$ condition, let $Y$ be a finite discrete space and $L_d$ be as in 
\eqref{LdDef}. We define for a function $H: \R \to \R$ the two operators
\begin{align}
\Psi_H (u)(y) &= \sum_{z \in Y} k(y,z) H \big(u(z)-u(y)\big),\quad y\in Y,\label{PsiUps}\\
B_H (u,v)(y) &= \sum_{z \in Y} k(y,z) H \big(u(z)-u(y) \big) \big(v(z)-v(y)\big),\quad y\in Y,
\nonumber
\end{align}
acting on functions $u,\,v:Y\to \R$.
Instead of the carré du champ operators $\Gamma(u)$ and $\Gamma_2(u)$ one considers the discrete operators $\Psi_\Upsilon (u)$ and 
\begin{equation}
\Psi_{2,\Upsilon}(u) = \frac{1}{2} \big(L_d \Psi_\Upsilon (u) -B_{\Upsilon^\prime}(u,L_d u)\big),
\end{equation}
where $\Upsilon (x) = e^x -1-x, \ x \in \R$. Note that replacing $\Upsilon$ with the quadratic function $H(x)=\frac{1}{2}x^2$,
this calculus leads to the carré du champ operators $\Gamma(u)$ and $\Gamma_2(u)$, that is $\Psi_H(u)=\Gamma(u)$ and 
$\Psi_{2,\Upsilon}(u)$ becomes $\Gamma_2(u)$. 

\begin{definition} \label{CDPSIDEF}
The operator $L_d$ from \eqref{LdDef} is said to satisfy the curvature-dimension condition $CD_\Upsilon (\kappa,d)$ at $y \in Y$ with $\kappa \in \R$ and $d \in [1,\infty]$, if for any function $u:Y\to \R$ there holds
\begin{equation}
\Psi_{2,\Upsilon} (u) (y) \geq \kappa \Psi_{\Upsilon} (u)(y) + \frac{1}{d} \big(L_d u (y) \big)^2.
\label{CDUngl}
\end{equation}
If $L_d$ satisfies $CD_\Upsilon (\kappa,d)$ at every $y \in Y$, then we say that $L_d$ satisfies $CD_\Upsilon (\kappa,d)$.
\end{definition}

\begin{bemerkung}
Note that in \cite[Definition 2.3]{We} a weaker condition has been introduced under the same name. More precisely, this condition is only based on estimate (\ref{CDUngl}) when $-L_d u >0$. Whenever $-L_d u\leq 0$, positivity of $\Psi_{2,\Upsilon} (u)$ is sufficient. See also \cite[Remark 2.4 (i)]{We} for further details. 

\end{bemerkung}

The hybrid curvature-dimension condition combines the $\Gamma$-calculus with the discrete $\Upsilon$-calculus. We now
consider the heat flow on the product space $X\times Y$ generated by the operator $L_c \oplus L_d$ where the Markov diffusion operator $L_c$ acts w.r.t.\ the (continuous) variable $x\in X$ and $L_d$ (given by \eqref{LdDef}) acts w.r.t.\ the discrete variable $y\in Y$.
The natural analogue of the carré du champ operator in the hybrid case is given by the operator $\Gamma \oplus \Psi_\Upsilon$,
where $\Gamma$ is the carré du champ operator associated with $L_c$ and $\Psi_\Upsilon$ is the counterpart for the discrete
variable. The corresponding iterated version is then defined in the following way (cf.\ \cite[Definition 7.1]{WZ}).

\begin{definition}
\label{DefGamma2}
For sufficiently regular functions $u : X \times Y \to \R$ we define
\begin{equation}
\left(\Gamma \oplus \Psi_{\Upsilon} \right)_2 (u) = \frac{1}{2} \Big( (L_c \oplus L_d) \left(\Gamma \oplus \Psi_\Upsilon \right)(u)-2 \Gamma \big(u,(L_c \oplus L_d) (u)\big) - B_{\Upsilon^\prime} \big(u,(L_c \oplus L_d) (u)\big)\Big).
\end{equation}
\end{definition}

We are now able to formulate the hybrid curvature-dimension condition. Recall that $\mu$ denotes the fixed invariant
reversible measure on $X$ of the semigroup generated by $L_c$.

\begin{definition}
\label{DefCDCond}
We say that the operator $L_c \oplus L_d$ satisfies the condition $CD_{hyb} (\kappa,d)$ for some $\kappa \in \R$ and $d \in [1,\infty]$ if for any sufficiently regular function $u: X\times Y \to \R$ there holds
\begin{equation}
\left( \Gamma \oplus \Psi_{\Upsilon} \right)_2 (u) (x,y) \geq \kappa \big(\Gamma \oplus \Psi_\Upsilon \big)(u)(x,y) + \frac{1}{d} \big((L_c \oplus L_d) (u) (x,y) \big)^2,
\end{equation}
for $\mu$-a.a.\ $x\in X$ and all $y\in Y$.
\end{definition}

The subsequent crucial inequality has already been proven in \cite[Lemma 7.2]{WZ}.

\begin{lemma}
\label{LemmaCrucInequ}
For sufficiently regular functions $u: X \times Y \to \R$ there holds
\begin{align*}
(\Gamma \oplus \Psi_\Upsilon)_2 (u) \geq (\Gamma_2 \oplus \Psi_{2,\Upsilon} )(u).
\end{align*}
\end{lemma}
%
This lemma is the key ingredient in the proof of the following highly important {\em hybrid tensorisation principle}.

\begin{satz}
\label{SatzHinrCD}
Let $L_c$ satisfy $CD(\kappa_1,n)$ and $L_d$ satisfy $CD_\Upsilon (\kappa_2,d)$ for some $\kappa_1,\kappa_2 \in \R$ and $n,d \in [1,\infty]$. Then the operator $L_c \oplus L_d$ satisfies $CD_{hyb} (\kappa,\tilde{d})$ with $\kappa = \min \{\kappa_1,\kappa_2\}$ and $\tilde{d} = n+d$. 
\end{satz}

\begin{proof}
By Lemma \ref{LemmaCrucInequ} and positivity of $\Gamma$ and $\Psi_\Upsilon$, we have
\begin{align*}
\left(\Gamma \oplus \Psi_{\Upsilon} \right)_2 (u) & \geq \Gamma_2 (u) + \Psi_{2,\Upsilon} (u) \geq \kappa_1 \Gamma (u) + \frac{1}{n} \big(L_c u \big)^2 + \kappa_2 \Psi_\Upsilon (u) + \frac{1}{d} \big(L_d u\big)^2 \\ & \geq \kappa (\Gamma \oplus \Psi_\Upsilon) (u) +\frac{1}{n} \big(L_c u \big)^2+\frac{1}{d} \big(L_d u\big)^2.
\end{align*}
Finally, using the elementary inequality $\frac{1}{c_1} x^2 +\frac{1}{c_2} y^2 \geq \frac{1}{c_1+c_2} (x+y)^2$ for $x,y \in \R$ and $c_1,c_2 >0$, we find that
\begin{align*}
\left(\Gamma \oplus \Psi_{\Upsilon} \right)_2 (u) \geq \kappa (\Gamma \oplus \Psi_\Upsilon) (u) + \frac{1}{n+d} \big(L_c u + L_d u)\big)^2 .
\end{align*}
\end{proof}

\begin{bemerkung}
Theorem \ref{SatzHinrCD} with $n=d=\infty$ was implicitely used in \cite[Theorem 7.3]{WZ} in order to derive a modified logarithmic Sobolev inequality in the hybrid setting. Indeed, the authors assumed $L_c$ and $L_d$ to satisfy $CD(\kappa_c, \infty)$ and $CD_\Upsilon (\kappa_d,\infty)$, respectively, and showed the inequality 
\begin{align*}
(\Gamma \oplus \Psi_\Upsilon)_2 (u) \geq \min \{\kappa_c,\kappa_d\} (\Gamma \oplus \Psi_\Upsilon)(u),
\end{align*}
which is exactly the condition $CD_{hyb} (\kappa,\infty)$ with $\kappa = \min \{\kappa_c,\kappa_d\}$.
\end{bemerkung}

\begin{korollar}
\label{KorollarCD}
Let $X = \R^n$. If $L_d$ satisfies $CD_\Upsilon (\kappa,d)$ with some $\kappa \geq 0$ and $d \in [1,\infty]$, then the operator $\mathcal{L} = \Delta \oplus L_d$ satisfies $CD_{hyb} (0,n+d)$.
\end{korollar}

\begin{proof}
This follows directly from Theorem \ref{SatzHinrCD} as the operator $\Delta$ satisfies $CD(0,n)$.
\end{proof}

\begin{bemerkung} \label{infiniteY}
Since our original interest lies in systems of finitely many reaction-diffusion equations, we restricted
the above presentation to the case of finite discrete spaces $Y$.
Note however that the definitions of the conditions $CD_\Upsilon (\kappa,d)$ and  $CD_{hyb} (\kappa,d)$
can be generalised to countable sets $Y$ by imposing appropriate conditions on the admissible
functions that ensure that the involved terms are finite, see e.g.\ \cite[Section 2]{WZ} for the condition $CD_\Upsilon (\kappa,d)$. Theorem \ref{SatzHinrCD} and Corollary \ref{KorollarCD} then easily
extend to this more general situation, which will be admitted later in the context of Ricci-flat graphs, see Theorem \ref{SatzRicciflat}. 
\end{bemerkung}

From now on, we will restrict ourselves to the important special case $L_c = \Delta$ on $X=\R^n$ (with $\mu$ being the Lebesgue measure). We will discuss several examples of operators $L_d$ (respectively graphs) which satisfy $CD_\Upsilon (0,d)$ with some $d\in [1,\infty)$. Corollary \ref{KorollarCD} then immediately shows that
 $\mathcal{L} = \Delta \oplus L_d$ satisfies the hybrid condition $CD_{hyb} (0,n+d)$.

\begin{beispiel}
\label{2-pkteGraph}
Let us consider the example from the introduction, now with an additional (symmetric) weight $\eta>0$, i.e.
\begin{equation} \label{firstexample}
\left\{
\begin{array}{r@{\;=\;}l@{\;}l}
\partial_t u_1 (t,x) - \Delta u_1 (t,x) & \eta\big(u_2(t,x) -u_1 (t,x)\big),\quad & t>0,\,x\in \iR^n, \\ 
 \partial_t u_2 (t,x) - \Delta u_2 (t,x) & \eta\big(u_1(t,x) -u_2 (t,x)\big),\quad & t>0,\,x\in \iR^n.
\end{array}
\right.
\end{equation}
Fix a two-element set $Y = \{y_1,y_2\}$ and set $u(t,x,y_1) = u_1(t,x)$ and $u(t,x,y_2) = u_2(t,x)$. 
Then, as already explained in the introduction, the system \eqref{firstexample} is equivalent to 
\begin{align*}
\partial_t u(t,x,y) - \Delta u(t,x,y) = L_d (t,x,y), \ (t,x,y) \in (0,\infty) \times \R^n \times Y,
\end{align*}
where $L_d$ is the discrete Laplacian \eqref{LdDef} on $Y$ generated by the kernel given by $k(y_1,y_2)=k(y_2,y_1)=\eta$. Setting $\tilde{y_1}=y_2$ and $\tilde{y_2}=y_1$, we find for $u: Y\to \R$ and $y\in Y$ that
\begin{align*}
2 \Psi_{2,\Upsilon} (u)(y) &= L_d \Psi_\Upsilon (u)(y) - B_{\Upsilon^\prime} (u,Lu)(y) \\ &= \eta \Big(\Psi_\Upsilon (u)(\tilde{y}) - \Psi_\Upsilon (u) (y)-\Upsilon^\prime(u(\tilde{y})-u(y)\big) \big(Lu(\tilde{y})-Lu(y)\big)\Big) \\ &= \eta^2 \Big( \Upsilon \big(u(y)-u(\tilde{y})\big) - \Upsilon\big(u(\tilde{y})-u(y)\big)-2\Upsilon^\prime(u(\tilde{y})-u(y)\big) (u(y)-u(\tilde{y}))\Big) \\&= \eta^2 \big(\Upsilon (-a) - \Upsilon (a) +2a\Upsilon^\prime(a)\big),
\end{align*}
with $a = u(\tilde{y})-u(y)$. Now, 
\begin{align*}
\Upsilon (-a)+\Upsilon (a) = e^{-a} +e^a-2 = \sum_{k=0}^\infty \frac{(-a)^k+a^k}{k!}-2 = a^2 + \sum_{k=3}^\infty \frac{(-a)^k+a^k}{k!} \geq a^2
\end{align*}
and by convexity of the exponential function
\begin{align*}
a\Upsilon^\prime (a) -\Upsilon (a) = e^a(a-1)+1 \geq -1+1 =0.
\end{align*}
Thus 
\begin{align*}
\Psi_{2,\Upsilon} (u)(y) = \frac{\eta^2}{2} \Big(\Upsilon (a)+\Upsilon(-a) +2\big(a\Upsilon^\prime(a)-\Upsilon(a)\big)\Big) \geq \frac{\eta^2}{2} a^2 = \frac{1}{2} \big(L_d u(y)\big)^2,
\end{align*}
which shows that $L_d$ satisfies $CD_\Upsilon (0,2)$. Together with Corollary \ref{KorollarCD}, we conclude that the operator $\mathcal{L} = \Delta \oplus L_d$ satisfies $CD_{hyb} (0,d)$ with $\displaystyle d = n+2$.
\end{beispiel}

As we will see in the following, the two-parameter family of functions 
\begin{equation}
\nu_{r,s} (w) = r \Upsilon^\prime (w)w + \Upsilon (-w) -s \Upsilon (w), \ w \in \R,
\label{DefNu}
\end{equation}
with constants $r,s \in \R$ appears in many examples at a crucial point when proving the validity of $CD_\Upsilon (0,d)$ with finite $d>0$ for some operator $L_d$. This observation has already been made in \cite{We}, where the author discusses several examples regarding the more general condition $CD_\Upsilon (0,F)$ with some $CD$-function $F$. There, also a few properties of the functions $\nu_{r,s}$ have been collected in the appendix. We do not repeat them here but add another property that has main importance for our calculations.

\begin{lemma}
\label{TechLemma}
Let $r \geq 0$. Then there holds
\begin{align*}
\nu_{r,r-1} (w) \geq C(r) w^2,\quad w\in \R,
\end{align*}
with $C(r): = \inf_{w \in \R} \frac{\nu_{r,r-1} (w)}{w^2} \geq 1$. Moreover, $C(r)$ is nondecreasing in $[0,\infty)$.
\end{lemma}

\begin{proof}
Similarly to the calculations in Example \ref{2-pkteGraph}, we find
\begin{align*}
\nu_{r,r-1} (w) = r\Upsilon^\prime(w)w +\Upsilon(-w)+\Upsilon(w)-r\Upsilon(w) \geq w^2.
\end{align*}
The second statement follows from the fact that $\Upsilon^\prime(w)w-\Upsilon(w)\ge 0$ for all $w\in \iR$.
\end{proof}

\begin{bemerkung}
\label{BemerkungMuench}
The function $\nu_{2,1}$ already appeared in M\"unch \cite{MUN} in a rather implicit way and was used to prove the validity of the condition $CD\psi (d,\kappa)$ with $\psi=\log$ on Ricci-flat graphs. To see this, note that the condition $CD_\Upsilon (\kappa,d)$ is equivalent to the condition $CD{\log} (d,\kappa)$ of M\"unch on unweighted graphs. This has been shown in \cite[Section 9]{WZ} for $d= \infty$ and can be easily generalized to the case $d<\infty$ as $L_d (\log u) = \Delta^{\log} u$, where $\Delta^{\log}$ is the nonlinear Laplace operator appearing in the dimension term in the $CD{\log}$ condition introduced in \cite{MUN}. For concave functions
$\psi\in C^1((0,\infty))$ Münch then considers the inequality
\begin{align}
\tilde{\psi} (x,y) &= \big(\psi^\prime (x) + \psi^\prime (y)\big)(1-xy) +y \Big(\psi(x)-\psi \big(\frac{1}{y}\big)\Big)+x\Big(\psi(y)-\psi \big(\frac{1}{x}\big)\Big) \notag\\ &\geq C_\psi \big( \psi(x) +\psi(y) -2\psi(1)\big)^2,\quad x,y>0,
\label{EstimateMuench}
\end{align}
where $C_{\psi} = \inf_{x,y>0} \frac{\tilde{\psi} (x,y)}{( \psi(x) +\psi(y) -2\psi(1))^2} \in [-\infty,\infty]$ cannot be calculated analytically in general. Now, for $\psi = \log$ we find
\begin{align*}
\tilde{\psi} (e^w,e^w) = \frac{2}{e^x} (1-e^{2x}) + 4x e^x = 2e^{-x} -2e^x+4xe^x = 2 \nu_{2,1} (w)
\end{align*}
and thus
\begin{align*}
\nu_{2,1} (w) = \frac{1}{2} \tilde{\psi} (e^w,e^w) \geq 2 C_{\log} w^2 = C(2) w^2
\end{align*} 
with $C(2) = 2C_{\log}$ being the optimal constant. Münch shows that $C_{\log} \geq \frac{1}{2}$ and gives a numerical approximation by $C_{\log} \approx 0,795$. See \cite[Chapter 6]{MUN} for more details. For the reader's convenience, we will give the proof that $CD_\Upsilon (0,d)$ is satisfied with some $d<\infty$ on Ricci-flat graphs later in this chapter.
\end{bemerkung}

The next example generalises our previous findings to the case of an unweighted complete graph with $m \in \N$ vertices. 

\begin{beispiel}
\label{BspVollstGraph}
Let $m \in \N\setminus \{1\}$ and $Y = \{y_1,\dots,y_m\}$. Let further $k(y_i,y_j) =1$ for $i \neq j$ and $k(y_i,y_i)=0$ for all $i \in \{1,\dots,m\}$. Then we have for $i \in \{1,\dots,m\}$ and $u:Y\to \R$
\begin{align*}
2 &\Psi_{2,\Upsilon} (u) (y_i) = L_d \Psi_\Upsilon (u)(y_i) - B_{\Upsilon^\prime} (u,L_du)(y_i) \\ &= \sum_{j\neq i} \Big(\Psi_{\Upsilon} (u)(y_j) - \Psi_\Upsilon (u)(y_i) - \Upsilon^\prime \big(u(y_j)-u(y_i)\big) \big(L_d u(y_j)-L_d u(y_i)\big)\Big). 
\end{align*}
Further, we have for $i\neq j$
\begin{align*}
L_d u(y_j)-L_d u(y_i) &= \sum_{l \neq j} \big(u(y_l)-u(y_j)\big) - \sum_{l \neq i} \big(u(y_l)-u(y_i)\big) \\ &= u(y_i)-(m-1)u(y_j)-u(y_j)+(m-1)u(y_i) \\
&= m \big(u(y_i)-u(y_j)\big),
\end{align*}
and thus
\begin{align*}
2 &\Psi_{2,\Upsilon} (u) (y_i) \\ &= \sum_{j\neq i} \Big( \Upsilon\big(u(y_i)-u(y_j)\big) - (m-1) \Upsilon \big(u(y_j)-u(y_i)\big) +m \Upsilon^\prime \big(u(y_j)-u(y_i)\big) \big(u(y_j)-u(y_i)\big)\Big) \\ &\qquad \qquad  + \sum_{j \neq i} \sum_{l \neq i,j} \Upsilon\big(u(y_l)-u(y_j)\big) \\ &\geq \sum_{j\neq i} \Big( \Upsilon\big(u(y_i)-u(y_j)\big) - (m-1) \Upsilon \big(u(y_j)-u(y_i)\big) +m \Upsilon^\prime \big(u(y_j)-u(y_i)\big) \big(u(y_j)-u(y_i)\big)\Big).
\end{align*}
Finally, using Lemma \ref{TechLemma} and Jensen's inequality we infer that
\begin{align*}
\Psi_{2,\Upsilon}(u)(y_i) \geq \frac{C(m)}{2} \sum_{j\neq i} \big(u(y_j)-u(y_i)\big)^2 &\geq \frac{C(m)}{2(m-1)} \Big(\sum_{j \neq i} \big(u(y_j)-u(y_i)\big)\Big)^2 \\ &= \frac{C(m)}{2(m-1)} \big(L_d u(y_i)\big)^2,
\end{align*}
i.e. $L_d$ satisfies $CD_\Upsilon \big(0,\frac{2(m-1)}{C(m)}\big)$. Together with Corollary \ref{KorollarCD}, we conclude that the corresponding operator $\mathcal{L} = \Delta \oplus L_d$ satisfies $CD_{hyb} (0,d)$ with $\displaystyle d = n+\frac{2(m-1)}{C(m)}$.
\end{beispiel}

\begin{bemerkung}
(i) Our result for the complete graph is closely related to \cite[Example 2.8]{We} in the special case when $\alpha=\frac{1}{2}$. In contrast to the findings there
(where $\alpha=\frac{1}{2}$ is excluded), we are able to treat this case as we do not have any interest in a positive curvature bound.

(ii) Note that in \cite{We} the author also considered the more general case of a weighted complete graph, see \cite[Example 2.7]{We}. There, he assumed the existence of a weight function $l: Y \to (0,\infty)$ such that the corresponding edge weights satisfy $k(y,z)=l(z)$ for all $y,z \in Y$ with $y\neq z$. Assuming symmetry for $k$ as we do, this implies that $k$ is constant on $\{(i,j)\in Y \times Y:i\neq j\}$, which by scaling time is equivalent to the case of an unweighted complete graph.
\end{bemerkung}

In the remaining part of this section we study the situation where the underlying graph to $L_d$ is Ricci-flat. This property covers many important examples of finite and locally finite but infinite graphs, such as the lattices $\iZ^d$. In fact, all (unweighted) Abelian Cayley graphs with degree $D\in \iN$ are Ricci-flat (see \cite{CY}), in particular unweighted finite complete graphs enjoy this property.
We already point out that for complete graphs the general $CD_\Upsilon$ result for Ricci-flat graphs, Theorem \ref{SatzRicciflat}, is weaker than that from Example \ref{BspVollstGraph}, see Example \ref{BeispielRicci} below.

The notion of a Ricci-flat graph was introduced in \cite{CY} and has been considered in many other works (see e.g. \cite{Harvard, DKZ, MUN, We, WZ}). We recall the definition here. Let $Y$ be a finite or an infinite countable set and $k: Y\times Y\to \{0,1\}$ be symmmetric. The egde set $E$ is the set of all pairs 
$(x,y)\in Y\times Y$ with $x\neq y$ and $k(x,y)=1$; in this case $x$ and $y$ are neighbours and we write $x\sim y$. For $x \in Y$ we set $N(x) = \{x\} \cup \{y \in Y: x \sim y\}$.
\begin{definition}
\label{DefinitionRicciflat}
Let $Y$ and $k$ be as before and suppose that $G = (Y,E)$ is a $D$-regular graph. Then $G$ is called Ricci-flat at $x$ if there exist maps $\eta_i:N(x) \mapsto Y, \ i=1,\dots,D$, such that the following properties are satisfied:
\begin{enumerate}
\item[(i)] $\eta_i (y) \sim y \text{ for all } i \in \{1,\dots,D\} \text{ and } y \in N(x)$,
\item[(ii)] $\eta_i (y) \neq \eta_j (y) \text{ for all } y \in N(x) \text{ and } i \neq j$,
\item[(iii)] $\displaystyle \bigcup_{j=1}^D \eta_i\big(\eta_j(x)\big) = \bigcup_{j=1}^D \eta_j \big(\eta_i(x)\big) \text{ for all } i \in \{1,\dots,D\}$.
\end{enumerate}
We call $G$ Ricci-flat if it is Ricci-flat at all $x \in Y$.
\end{definition}

The following lemma collects basic properties of Ricci-flat graphs and has been taken from \cite[Lemma 6.3]{MUN}, see also \cite[Lemma 3.16]{DKZ}.

\begin{lemma}
\label{LemmaPropRicci}
Let $G=(Y,E)$ be a $D$-regular graph which is Ricci-flat at $x \in Y$. Let $\eta_1,\dots,\eta_D$ be the maps from Definition \ref{DefinitionRicciflat}. Then the following holds true:
\begin{enumerate}
\item[(i)] For any function $u: Y \to \R$ and for all $i \in \{1,\dots,D\}$ we have
\begin{equation}
\sum_{j=1}^D u\big(\eta_i(\eta_j(x)\big) = \sum_{j=1}^D u\big(\eta_j(\eta_i(x)\big).
\label{PropRicciflat}
\end{equation}
\item[(ii)] For every $i \in \{1,\dots,D\}$ there exists a unique $i^* \in \{1,\dots,D\}$ such that $\eta_i\big(\eta_{i^*}(x)\big)=x$. Moreover, the map $i \mapsto i^*$ is a permutation of $\{1,\dots,D\}$.
\end{enumerate}
\end{lemma}

The following theorem shows that $\mathcal{L}=\Delta \oplus L_d$ satisfies $CD_{hyb} (0,d)$ with some $d\in [1,\infty)$ whenever the graph corresponding to $L_d$ is Ricci-flat. We point out that here $Y$ can be also an infinite countable set
and remind the reader of Remark \ref{infiniteY}. Since the graph is locally finite, all involved sums
are finite and thus no further conditions on the admissible functions are required.

\begin{satz}
\label{SatzRicciflat}
Let the kernel $k$ be chosen in such a way that the underlying graph is a $D$-regular Ricci-flat graph with vertex set $Y$. Let $L_d$ denote the operator generated by the kernel $k$. Then $L_d$
satisfies the condition $CD_\Upsilon(0,\frac{2D}{C(2)})$, where 
$C(2)>1$ is the constant from Lemma \ref{TechLemma}. Moreover,
the operator $\mathcal{L}=\Delta \oplus L_d$ satisfies $CD_{hyb} (0,d)$ with $d = n+\frac{2D}{C(2)}$.
\end{satz}

The main part of the proof, which consists of the proof of the condition $CD_\Upsilon (0,\frac{2D}{C(2)})$ for the discrete operator $L_d$, is essentially contained in the proof of \cite[Theorem 6.6]{MUN} as $CD_\Upsilon (0,d)$ is equivalent to Münch's $CD \log (d,0)$-condition. Nevertheless, as Münch's notation cannot be translated so easily to our special case, we will give the proof here for the 
reader's convenience.

\begin{proof}
Let $x \in Y$ and $u:\,Y\rightarrow \iR$. We use the notation $z=u(x)$, $z_j = u(\eta_j(x))$ and $z_{ij} = u\big(\eta_j (\eta_i (x))\big)$ for $i,j \in \{1,\dots,D\}$. We have
\begin{align*}
2  \Psi_{2,\Upsilon}& (u)(x) =  L_d \Psi_\Upsilon (u)(x) - B_{\Upsilon^\prime} (u,L_d u)(x) \\ &= \sum_{i=1}^D \Big( \Psi_\Upsilon (u)(\eta_i(x)) - \Psi_\Upsilon (u)(x) - \Upsilon^\prime \big(u(\eta_i(x))-u(x)\big) \big(L_d u(\eta_i(x))-L_d u(x)\big)\Big) \\ &= \sum_{i=1}^D \sum_{j=1}^D \Big( \Upsilon(z_{ij}-z_i)-\Upsilon(z_j-z)-\Upsilon^\prime (z_i-z)(z_{ij}-z_i-z_j+z)\Big) \\ &= \sum_{j=1}^D \sum_{i=1}^D \Big( \Upsilon(z_{ij}-z_i)-\Upsilon(z_j-z)-\Upsilon^\prime (z_j-z)(z_{ij}-z_i-z_j+z)\Big),
\end{align*}
where we used first that property (\ref{PropRicciflat}) yields
\begin{align*}
\sum_{j=1}^D \Upsilon^\prime (z_i-z)z_{ij} = \sum_{j=1}^D \Upsilon^\prime (z_i-z)z_{ji}
\end{align*}
and then symmetry, together with Fubini's theorem. Now, as
\begin{align*}
\Upsilon (a) - \Upsilon (b) - \Upsilon^\prime (b)(a-b) = e^b \Upsilon (a-b), \ a,b \in \R,
\end{align*}
we find 
\begin{align*}
2 \Psi_{2,\Upsilon}& (u)(x) = \sum_{j=1}^D e^{z_j-z} \sum_{i=1}^D \Upsilon (z_{ij}-z_i-z_j+z) \\ &= \sum_{j=1}^D e^{z_j-z} \Big( \Upsilon (z_{j^*j}-z_{j^*}-z_j+z)+\sum_{i=1, i\neq j^*}^D  \Upsilon (z_{ij}-z_i-z_j+z) \Big) \\ &= \sum_{j=1}^D e^{z_j-z} \Big( \Upsilon (2z-z_{j^*}-z_j)+\sum_{i=1, i\neq j^*}^D  \Upsilon (z_{ij}-z_i-z_j+z) \Big)\geq \sum_{j=1}^D e^{z_j-z} \Upsilon (2w_j),
\end{align*}
with $w_j = z-\frac{1}{2}z_j-\frac{1}{2}z_{j^*}$. Here we applied Lemma \ref{LemmaPropRicci} (ii) and the positivity of $\Upsilon$. We next apply the rearrangement inequality similarly to the proof of \cite[Theorem 3.17]{DKZ}, i.e. we use that for all permutations $\pi$ on $\{1,\dots,D\}$ and all $0 \leq a_1\leq a_2 \leq \dots \leq a_D$ and all $0\leq b_1 \leq b_2 \leq \dots \leq b_D$ there holds
\begin{align*}
\sum_{j=1}^D a_{\pi(j)} b_j \geq \sum_{j=1}^D a_{D+1-j} b_j.
\end{align*}
In this sense, assume w.l.o.g. that $z_1 \leq z_2 \leq \dots \leq z_D$ and set $j^\prime = D+1-j$ for $j = 1,\dots,D$. Then
\begin{align*}
2 \Psi_{2,\Upsilon} (u)(x) &= \sum_{j=1}^D e^{z_j-z} \big(e^{2w_j}-2w_j-1\big) \\ &= \sum_{j=1}^D e^{z-z_{j^*}} - \sum_{j=1}^D e^{z_j-z} \big(2z-z_j-z_{j^*}\big) - \sum_{j=1}^D e^{z_j-z} \\ &\geq \sum_{j=1}^D e^{z-z_{j^\prime}} - \sum_{j=1}^D e^{z_j-z} \big(2z-z_j-z_{j^\prime}\big) - \sum_{j=1}^D e^{z_j-z} \\ &= \sum_{j=1}^D e^{z_j-z} \big(e^{2z-z_j-z_{j^\prime}}-(2z-z_j-z_{j^\prime})-1\big) = \sum_{j=1}^D e^{z_j-z} \Upsilon (2\tilde{w}_j)
\end{align*}
with $\tilde{w}_j = z-\frac{1}{2}z_j-\frac{1}{2}z_{j^\prime}$. The advantage of this estimate lies in the fact that $\tilde{w}_j = \tilde{w}_{j^\prime}$ holds as $(j^\prime)^\prime =j$. Finally, using Lemma \ref{TechLemma} and the convexity and nonnegativity of the exponential function and the square function yields
\begin{align*}
\Psi_{2,\Upsilon} (u)(x) & = \frac{1}{4} \sum_{j=1}^D e^{z_j-z} \Upsilon (2\tilde{w}_j) + \frac{1}{4} \sum_{j=1}^D e^{z_{D+1-j}-z} \Upsilon (2\tilde{w}_{D+1-j}) \\ &= \frac{1}{4} \sum_{j=1}^D \big(e^{z_j-z} +e^{z_{j^\prime}-z}\big) \Upsilon (2\tilde{w}_j) \geq \frac{1}{2} \sum_{j=1}^D e^{-\tilde{w}_j} \Upsilon (2\tilde{w}_j) \geq \frac{C(2)}{2} \sum_{j=1}^D \tilde{w}_j^2 \\ &\geq \frac{C(2)}{2D} \Big(\sum_{j=1}^D \tilde{w}_j\Big)^2 = \frac{C(2)}{2D} \Big(L_d u(x)\Big)^2,
\end{align*}
since $e^{-a} \Upsilon (2a) = \nu_{2,1} (-a)$ and 
\begin{align*}
\sum_{j=1}^D \tilde{w}_j = \sum_{j=1}^D (z-\frac{1}{2}z_j-\frac{1}{2} z_{D+1-j})= \sum_{j=1}^D (z-z_j) = -L_d u(x).
\end{align*}
As $x \in Y$ was arbitrary, it follows that $L_d$ satisfies $CD_\Upsilon (0,\frac{2D}{C(2)})$ and thus Theorem \ref{SatzHinrCD} shows that the corresponding operator $\mathcal{L}$ satisfies $CD_{hyb} (0,d)$ with $d=n+\frac{2D}{C(2)}$.
\end{proof}

\begin{beispiel}
\label{BeispielRicci}
(i) \emph{Complete graphs.} As already mentioned, unweighted finite complete graphs are Ricci-flat. 
Therefore, Theorem \ref{SatzRicciflat} shows that in case of an unweighted complete graph with 
$m$ vertices the operator $\mathcal{L} = \Delta \oplus L_d$ with $L_d$ being the operator generated by the kernel $k$ satisfies $CD_{hyb} (0,d)$ with $d=n+\frac{2(m-1)}{C(2)}$. For $m>2$, this result is weaker than the one from Example \ref{BspVollstGraph}, which establishes $CD_{hyb} (0,d)$ with the smaller number $d=n+\frac{2(m-1)}{C(m)}$ (recall that $C(m)>C(2)$ for $m>2$).

(ii) \emph{Complete bipartite graphs.} Let $k$ be chosen in such a way that the underlying graph is a ($D$-regular) complete bipartite graph. In the appendix of \cite{CKK} it has been shown that such a graph is Ricci-flat. Thus, it follows that the operator $L_d$ generated by $k$ satisfies $CD_\Upsilon (0,\frac{2D}{C(2)})$ and that $\mathcal{L}$ satisfies $CD_{hyb} (0,d)$ with $d = n+\frac{2D}{C(2)}$. An example for such a graph is given by the square, i.e.\ the graph corresponding to the vertex set $Y=\{y_1,y_2,y_3,y_4\}$ and the (symmetric) kernel $k$ with $k(y_i,y_j)=1$ for $\vert i-j \vert \in \{1,3\}$ and $k=0$ else.

(iii) \emph{Infinite graphs including the lattice $\Z$.} Although we focus on finite graphs in this article, we emphasise that Theorem \ref{SatzRicciflat} also shows the validity of $CD_\Upsilon (0,d)$ for certain infinite graph structures, recall Remark \ref{infiniteY}. Consider for example the graph Laplacian on the lattice $\Z$, i.e.\ the operator $L_d$ generated by the kernel $k: \Z \times \Z \to \{0,1\}$ with $k(x,y) =1$ if and only if $\vert x-y \vert = 1$. It can be easily seen that the corresponding graph is 2-regular and Ricci-flat. It follows that $L_d$ satisfies $CD_\Upsilon \big(0,\frac{4}{C(2)}\big)$. 
\end{beispiel}

\section{Li-Yau inequality}

In this section we derive Li-Yau inequalities for positive solutions of the heat equation with
operator $\mathcal{L} = \Delta \oplus L_d$ on $\iR^n\times Y$, where $L_d$ is as above and $Y$
is a finite set. In what follows the operators $\Gamma, \Gamma_2$ and $\nabla$ always act w.r.t.\ the continuous variable $x\in \iR^n$, while the operators $\Psi_H, \Psi_{2,\Upsilon}$ and $B_H$ refer to  the discrete variable $y\in Y$. 

By saying that $u$ is a {\em smooth} function on $(0,\infty) \times \Omega \times Y$, where $\Omega\subset \iR^n$ is
open, we mean that for every $y\in Y$ the function $u(\cdot,\cdot,y)$ is $C^\infty$ smooth
on $(0,\infty)\times \Omega$. Note that classical solutions of the equation $\partial_t u - \mathcal{L} u= 0$  on $(0,\infty) \times \Omega \times Y$ are always smooth in this sense, by parabolic
regularity theory.

The first step is to derive an evolution equation for the function $v = \log u$. This equation will be
the starting point of our further calculations.

\begin{lemma}
\label{LemmaEvolEqu}
Let $\Omega$ be an open subset of $\iR^n$ and $u$ be a positive smooth solution of the equation $\partial_t u - \mathcal{L} u= 0$ on $(0,\infty) \times \Omega \times Y$. Then the function $v = \log u$ satisfies
\begin{equation}
\partial_t v -\mathcal{L} v = \big( \vert \nabla v \vert^2 + \Psi_\Upsilon (v)\big)
\quad\mbox{on}\;\;(0,\infty) \times \Omega \times Y.
\end{equation}
\end{lemma}

\begin{proof}
From \cite[Lemma 2.2]{WZ} we know that
\begin{align*}
L_d \log u = \frac{L_d u}{u} - \Psi_\Upsilon (\log u).
\end{align*}
This yields
\begin{align*}
\mathcal{L} \log u &= \Delta \log u + L_d \log u = \nabla \cdot \Big( \frac{\nabla u}{u}\Big) + \frac{L_d u}{u} - \Psi_\Upsilon (\log u) \\ &= \frac{\Delta u}{u} - \frac{\vert \nabla u \vert^2}{u^2} + \frac{L_d u}{u} - \Psi_\Upsilon (\log u) = \frac{\mathcal{L} u}{u} - \vert \nabla \log u \vert^2 -\Psi_\Upsilon (\log u) \\ &= \partial_t \log u - \big(\vert \nabla \log u \vert^2 +\Psi_\Upsilon (\log u)\big).
\end{align*}
\end{proof}

For the proof of the Li-Yau inequality we will need the following proposition.

\begin{proposition}
\label{PropositionLY}
Let $\Omega$ be an open subset of $\iR^n$ and $u$ be a positive smooth solution of the equation $\partial_t u - \mathcal{L} u= 0$ on $(0,\infty) \times \Omega \times Y$ and suppose that $\mathcal{L}$ satisfies $CD_{hyb} (0,d)$ for some $d\in [1,\infty)$. Let $\theta \in (0,1)$ be fixed, $v=\log u$ and define the function $G: (0,\infty) \times \Omega \times Y \to \R$ by
\begin{equation}
G(t,x,y) = -t \Big( \theta \big(\vert \nabla v \vert^2 + \Psi_\Upsilon (v)\big) + \mathcal{L} v \Big),\quad t\in (0,\infty),\,x\in \Omega,\,y\in Y.
\label{DefG}
\end{equation}
Then we have
\begin{align*}
\partial_t G - \mathcal{L} G \leq  2 \langle \nabla v, \nabla G \rangle + \frac{G}{t} - &\frac{2(1-\theta)}{d} \frac{G^2}{t} -\frac{4\theta(1-\theta)}{d} \big( \vert \nabla v \vert^2 + \Psi_\Upsilon (v)\big) G +B_{\Upsilon^\prime} (v,G)
\end{align*}
on $(0,\infty) \times \Omega \times Y$.
\end{proposition}

\begin{proof}
By Definition \ref{DefGamma2} we have
\begin{align*}
\mathcal{L} \big(\vert \nabla v \vert^2 + \Psi_\Upsilon (v)\big) &= \mathcal{L} \big( \Gamma \oplus \Psi_\Upsilon\big) (v) = 2 (\Gamma \oplus \Psi_\Upsilon)_2 (v) + 2 \Gamma ( v,\mathcal{L} v)+ B_{\Upsilon^\prime} (v,\mathcal{L} v).
\end{align*}
Further, by Lemma \ref{LemmaEvolEqu}, equation (\ref{DefG}) can be rewritten as 
\begin{align}
G(t,x,y) &= t \Big( (1-\theta) \big(\vert \nabla v \vert^2 + \Psi_\Upsilon (v)\big) - \partial_t v \Big) \notag \\ &= -t \big( \theta \partial_t v +(1-\theta) \mathcal{L} v \big).
\label{AlternativeDarstG}
\end{align}
This yields
\begin{align*}
-\mathcal{L} G &= -t(1-\theta) \mathcal{L} \big(\vert \nabla v \vert^2 + \Psi_\Upsilon (v)\big) + t \partial_t \mathcal{L} v \\ &= -t(1-\theta) \big(2 (\Gamma \oplus \Psi_\Upsilon)_2 (v) + 2 \Gamma ( v,\mathcal{L} v)+ B_{\Upsilon^\prime} (v,\mathcal{L} v)\big) + t \partial_t \mathcal{L} v.
\end{align*}
As $\mathcal{L}$ satisfies $CD_{hyb} (0,d)$ we have
\begin{align*}
-\mathcal{L} G &\leq -\frac{2 t(1-\theta)}{d} \big(\mathcal{L} v\big)^2 - 2t(1-\theta) \Gamma ( v,\mathcal{L} v)- t (1-\theta) B_{\Upsilon^\prime} (v,\mathcal{L} v) + t \partial_t \mathcal{L} v.
\end{align*}
We now use that (\ref{DefG}) and (\ref{AlternativeDarstG}) are equivalent to 
\begin{align*}
&\mathcal{L} v = -\frac{G}{t} - \theta \big(\vert \nabla v \vert^2 + \Psi_\Upsilon (v)\big) \text{ and } \mathcal{L} v = -\frac{G}{(1-\theta) t} - \frac{\theta}{1-\theta} \partial_t v.
\end{align*}
This gives
\begin{align*}
-\mathcal{L} G &\leq -\frac{2 t(1-\theta)}{d} \Big(\frac{G}{t} + \theta \big(\vert \nabla v \vert^2 + \Psi_\Upsilon (v)\big)\Big)^2 + 2t \big\langle \nabla v,\nabla \Big(\frac{G}{t} +\theta\partial_t v \Big) \big\rangle\\ &\qquad   + t \partial_t \Big( -\frac{G}{t} - \theta \big( \vert \nabla v \vert^2 + \Psi_\Upsilon (v)\big) \Big) - t (1-\theta) B_{\Upsilon^\prime} (v,\mathcal{L} v) \\ &\leq -\frac{2 t(1-\theta)}{d} \Big(\frac{G^2}{t^2} + 2\frac{\theta G}{t} \big( \vert \nabla v \vert^2 + \Psi_\Upsilon (v)\big)\Big) + 2t \big\langle \nabla v,\nabla \Big(\frac{G}{t} +\theta\partial_t v \Big) \big\rangle \\ &\qquad    - t \Big( \frac{(\partial_t G) t - G}{t^2} + \theta \partial_t \big( \vert \nabla v \vert^2 + \Psi_\Upsilon (v)\big) \Big) - t (1-\theta) B_{\Upsilon^\prime} (v,\mathcal{L} v) \\ &= 2 \langle \nabla v,\nabla G \rangle + 2\theta t \langle \nabla v, \nabla \partial_t v \rangle + \frac{G}{t} -\frac{2(1-\theta)}{d} \frac{G^2}{t} -\frac{4\theta(1-\theta)}{d} \big( \vert \nabla v \vert^2 + \Psi_\Upsilon (v)\big) G \\ &\qquad  - \partial_t G - 2 \theta t \langle \nabla v, \nabla \partial_t v \rangle - t\theta \partial_t \Psi_\Upsilon(v)  - t (1-\theta) B_{\Upsilon^\prime} (v,\mathcal{L} v).
\end{align*}
Using the Cauchy-Schwarz inequality and the fact that
\begin{align*}
\partial_t \Psi_\Upsilon (v) (t,x,y) &= \sum_{z \in Y} k(y,z) \Upsilon^\prime \big((v(t,x,z)-v(t,x,y)\big) \big(\partial_t v(t,x,z) - \partial_t v(t,x,y)\big) \\ &= B_{\Upsilon^\prime} (v,\partial_t v)(t,x,y)
\end{align*}
finally gives
\begin{align*}
\partial_t G-\mathcal{L} G &\leq 2 \langle \nabla v,\nabla G \rangle + \frac{G}{t} -\frac{2(1-\theta)}{d} \frac{G^2}{t} -\frac{4\theta(1-\theta)}{d} \big( \vert \nabla v \vert^2 + \Psi_\Upsilon (v)\big) G\\ &\qquad    - t\theta B_{\Upsilon^\prime} (v,\partial_t v) - t (1-\theta) B_{\Upsilon^\prime} (v,\mathcal{L} v) \\ &=  2 \langle \nabla v,\nabla G \rangle + \frac{G}{t} -\frac{2(1-\theta)}{d} \frac{G^2}{t} -\frac{4\theta(1-\theta)}{d} \big(\vert \nabla v \vert^2 + \Psi_\Upsilon (v)\big) G \\ &\qquad  + B_{\Upsilon^\prime} \big(v,- t\theta \partial_t v- t (1-\theta)\mathcal{L} v\big) \\ &= 2 \langle \nabla v,\nabla G \rangle + \frac{G}{t} -\frac{2(1-\theta)}{d} \frac{G^2}{t} -\frac{4\theta(1-\theta)}{d} \big(\vert \nabla v \vert^2 + \Psi_\Upsilon (v)\big) G + B_{\Upsilon^\prime} (v,G),
\end{align*}
where we used (\ref{AlternativeDarstG}) in the last step.
\end{proof}

We are now able to prove a Li-Yau inequality for positive solutions to the diffusion equation 
$\partial_t u - \mathcal{L} u= 0$. Note that in the following theorem we consider functions
that merely solve the equation locally w.r.t.\ the continuous variable. In the sequel, for $x\in \iR^n$
and $r>0$ we let $B_r (x) = \{y \in \R^n: \vert y-x \vert < r\}$.

\begin{satz}
\label{SatzLY}
Let $\theta\in (0,1)$, $\rho >0$, $z \in \R^n$ and let $u: (0,\infty) \times B_{3 \rho}(z) \times Y \to (0,\infty)$ be a smooth solution to the equation $\partial_t u - \mathcal{L} u= 0$ on $(0,\infty) \times B_{3\rho} (z) \times Y$. Assume that there exists some $d\in [1,\infty)$ such that the operator $\mathcal{L}$ satisfies $CD_{hyb} (0,d)$. Then there exists a constant $C>0$ which is independent of $u,\rho,z$ and $\theta$ such that the function $v = \log u$ satisfies
\begin{equation}
(1-\theta) \big( \vert \nabla v \vert^2 + \Psi_\Upsilon (v)\big) - \partial_t v \leq \frac{d}{2t(1-\theta)} + \frac{C}{2(1-\theta)\rho^2} \Big(1+\frac{1}{\theta(1-\theta)}\Big)
\label{LYInequ}
\end{equation}
on $(0,\infty) \times B_{\rho} (z) \times Y$.
\end{satz}

\begin{bemerkung}
Note that by Lemma \ref{LemmaEvolEqu}, the left-hand side in \eqref{LYInequ} can be rewritten as
\[
(1-\theta) \big( \vert \nabla v \vert^2 + \Psi_\Upsilon (v)\big) - \partial_t v=-\mathcal{L}v-
\theta  \big( \vert \nabla v \vert^2 + \Psi_\Upsilon (v)\big).
\]
\end{bemerkung}

The procedure of our proof is much inspired by the argument given in \cite[Theorem 12.2]{Li} in the purely continuous
setting. 

\begin{proof}
Note first that we may assume without restriction of generality that $u$ is a smooth function on
$[0,\infty) \times B_{3 \rho}(z) \times Y$. In fact, for any $\varepsilon>0$, we can restrict $u$
to $[\varepsilon,\infty) \times B_{3 \rho}(z) \times Y$ and shift the time by setting 
$\tilde{t}=t-\varepsilon$. The function $\tilde{u}$ defined by $\tilde{u}(\tilde{t},x,y)=u(t,x,y)$
is then smooth on $[0,\infty) \times B_{3 \rho}(z) \times Y$ and satisfies
$\partial_t \tilde{u} - \mathcal{L} \tilde{u}= 0$, too. Having proved the claimed estimate for $\tilde{u}$, we
can transform back to $u$ resulting in an estimate of the form \eqref{LYInequ}, but valid 
only for $t>\varepsilon$ and with $\frac{1}{t-\varepsilon}$ in place of $\frac{1}{t}$. 
Since  $\varepsilon>0$ can be made arbitrary small, we obtain the desired estimate for $u$ by sending
$\varepsilon\to 0$.

We choose a cutoff function $\Phi \in C^2 \big(B_{3\rho}(z)\big)$ with respect to the continuous variable of the form 
\begin{align*}
\Phi (x) = \begin{cases}
1 , \text{ if } 0 \leq \vert x-z \vert < \rho,\\
\Lambda (\vert x-z\vert), \text{ if } \rho \leq \vert x-z \vert \leq 2\rho,\\
0, \text{ else,}\end{cases}
\end{align*}
where the $C^2$-function $\Lambda: [\rho,2\rho] \to [0,1]$ is such that $\Lambda(r)>0$ in $[\rho,2\rho)$,
\begin{align}
&-\frac{C_1}{\rho} \sqrt{\Lambda (r)} \leq \Lambda^\prime (r) \leq 0
\text{ and } \Lambda^{\prime \prime} (r) \geq -\frac{C_1}{\rho^2},\quad r\in[\rho,2\rho],
\label{BedPhi}
\end{align}
with some constant $C_1>0$ independent of $\rho$. 

Let $T>0$ be arbitrarily fixed and the function $G$ be given as in Proposition \ref{PropositionLY}, but now considered on the domain $[0,T] \times B_{3\rho} (z) \times Y$, i.e.
\begin{align*}
G(t,x,y) = t \Big( (1-\theta) \big(\vert \nabla v(t,x,y) \vert^2 + \Psi_\Upsilon (v)(t,x,y)\big) - \partial_t v(t,x,y) \Big).
\end{align*}
Our aim is to show \eqref{LYInequ} (where the left-hand side is precisely $\frac{1}{t}G(t,x,y)$) with $t=T$. 

Clearly, the support of the function $(\Phi G): [0,T] \times B_{3\rho} (z) \times Y \to \R, \ (\Phi G)(t,x,y) = \Phi (x) G(t,x,y)$ is 
contained in $[0,T] \times \bar{B}_{2\rho} (z) \times Y$. Let $(t_*,x_*,y_*)$ be a maximum point of $\Phi G$ considered as 
a function on $[0,T] \times \bar{B}_{2\rho} (z) \times Y$ (which exists due to compactness of $\bar{B}_{2\rho} (z)$ in $\R^n$ and finiteness of $Y$).
If $(\Phi G) (t_*,x_*,y_*) \le 0$, then for all $(x,y) \in B_{ \rho} (z) \times Y$ we have
\begin{align*}
G (T,x,y) = (\Phi G) (T,x,y) \leq  (\Phi G) (t_*,x_*,y_*) \leq 0,
\end{align*}
and thus $\frac{1}{T}G(T,x,y)\le 0$ in $B_{ \rho} (z) \times Y$, which implies \eqref{LYInequ} with $t=T$ (by positivity of the
right-hand side in \eqref{LYInequ}).

Suppose now that $(\Phi G) (t_*,x_*,y_*) >0$ . Then $(t_*,x_*,y_*) \in (0,T] \times B_{2\rho} (z) \times Y$, by definition
of $G$ and $\Phi$, and therefore at the maximum point $(t_*,x_*,y_*)$
\begin{equation}
\nabla (\Phi G) = 0, \quad \Delta (\Phi G) \leq 0\; \text{ and }\; \partial_t (\Phi G) = \Phi \partial_t G \geq 0 
\label{MaxBedKont}
\end{equation}
must hold. Furthermore, there holds 
\begin{equation}
(\Phi G) (t_*,x_*,y_*) \geq (\Phi G) (t_*,x_*,z) , \ z \in Y.
\label{MaxBedDisc}
\end{equation}

For $x \in B_{2\rho} (z)$ we will need a lower bound for $\Delta \Phi (x)$. By (\ref{BedPhi}), we find for $x \in \bar{B}_{2\rho} (z) \setminus B_{\rho} (z)$
\begin{align*}
\Delta \Phi (x) &= \Delta \Lambda (\vert x-z \vert) = \Lambda^{\prime \prime} (\vert x-z \vert) \big \vert \nabla \vert x-z \vert \big \vert ^2 + \Lambda ^\prime (\vert x-z \vert) \Delta \vert x-z \vert \\&= \Lambda^{\prime \prime} (\vert x-z \vert) + \Lambda ^\prime (\vert x-z \vert) \frac{n-1}{\vert x-z \vert} \geq -\frac{C_1}{\rho^2} - \frac{C_1}{\rho} \sqrt{\Lambda (\vert x-z \vert)} \frac{n-1}{\vert x-z \vert} \\ &\geq -\frac{C_1}{\rho^2} - \frac{C_1}{\rho} \sqrt{\Lambda (\vert x-z \vert)} \frac{n-1}{2 \rho } \geq -\frac{C_2}{\rho^2},
\end{align*}
where $C_2 = C_1 \big(1+\frac{n-1}{2}\big)$. Note that this estimate also holds for any $x \in B_{ \rho} (z)$, by negativity of the right-hand side. Now, since $\Phi$ does not depend on the discrete variable $y$, we find that
\begin{align*}
\mathcal{L} (\Phi G) = \Delta (\Phi G) + \Phi L_d (G) &= (\Delta \Phi) G + 2 \langle \nabla \Phi, \nabla G \rangle + \Phi (\mathcal{L} G).
\end{align*}
Together with the previous estimate, this yields on $(0,T] \times B_{2\rho} (z) \times Y$
\begin{align*}
-\mathcal{L} (\Phi G) \leq \frac{C_2}{\rho^2} |G| - 2 \langle \nabla \Phi, \nabla G \rangle - \Phi (\mathcal{L} G).
\end{align*}
Applying Proposition \ref{PropositionLY} it follows that on $(0,T] \times B_{2\rho} (z) \times Y$ we have that
\begin{align*}
-\mathcal{L} (\Phi G) \leq \frac{C_2}{\rho^2} |G| -& \frac{2}{\Phi} \langle \nabla \Phi, \Phi \nabla G \rangle + \Phi \Big(- \partial_t G + 2 \langle \nabla v, \nabla G \rangle + \frac{G}{t} - \frac{2(1-\theta)}{d} \frac{G^2}{t} \\ &-\frac{4\theta(1-\theta)}{d} \big( \vert \nabla v \vert^2 + \Psi_\Upsilon (v)\big) G +B_{\Upsilon^\prime} (v,G)\Big).
\end{align*}

We now evaluate this inequality at the maximum point $(t_*,x_*,y_*)$ and apply (\ref{MaxBedKont}). In particular, observe that 
$\Phi \nabla G = -G \nabla \Phi$ at $(t_*,x_*,y_*)$. Hence, we have at the maximum point 
\begin{align}
-L_d (\Phi G) &\leq \frac{C_2}{\rho^2} G + \frac{2}{\Phi} G \langle \nabla \Phi,  \nabla \Phi \rangle - 2 G \langle \nabla v, \nabla \Phi \rangle + \frac{\Phi G}{t_*} - \frac{2(1-\theta)}{d} \frac{\Phi G^2}{t_*}
\nonumber \\ &\qquad   -\frac{4\theta(1-\theta)}{d} \big( \vert \nabla v \vert^2 + \Psi_\Upsilon (v)\big) \Phi G + B_{\Upsilon^\prime} (v,\Phi G) \nonumber\\ &\leq \frac{C_2}{\rho^2} G + \frac{2}{\Phi} \vert \nabla \Phi \vert^2 G + 2 G  \vert \nabla v \vert \vert \nabla \Phi \vert + \frac{\Phi G}{t_*} - \frac{2(1-\theta)}{d} \frac{\Phi G^2}{t_*} 
\nonumber\\ &\qquad   -\frac{4\theta(1-\theta)}{d} \big( \vert \nabla v \vert^2 + \Psi_\Upsilon (v)\big) \Phi G + B_{\exp} (v,\Phi G) - L_d (\Phi G),\label{eststep}
\end{align}
where we used the Cauchy-Schwarz inequality in the second step. By (\ref{BedPhi}), on $B_{2 \rho} (z)\setminus B_{ \rho} (z)$ there holds
\begin{align*}
\vert \nabla \Phi (x) \vert \leq \big \vert \Lambda^\prime (\vert x-z \vert) \big \vert \leq \frac{C_1}{\rho} \sqrt{\Phi (x)}
\end{align*}
and thus $|\nabla \Phi (x)|\le  \frac{C_1}{\rho} \sqrt{\Phi (x)}$ for all $x\in B_{2 \rho} (z)$. Using this estimate, 
positivity of $\Psi_\Upsilon (v)$ and the fact that \eqref{MaxBedDisc} yields 
\begin{align*}
B_{\exp} (v,\Phi G) (t_*,x_*,y_*) &= \sum_{z \in Y} k(y_*,z) \frac{u(t_*,x_*,z)}{u(t_*,x_*,y_*)} \big( (\Phi G) (t_*,x_*,z)-(\Phi G) (t_*,x_*,y_*)\big) \leq 0,
\end{align*}
we further deduce from \eqref{eststep} that
\begin{align*}
0 &\leq \frac{C_3}{\rho^2} G + \frac{2 C_1}{\rho} \Phi^\frac{1}{2} G  \vert \nabla v \vert + \frac{\Phi G}{t_*} - \frac{2(1-\theta)}{d} \frac{\Phi G^2}{t_*} -\frac{4\theta(1-\theta)}{d} \vert \nabla v \vert^2 \Phi G
\end{align*}
with $C_3 = C_2+2 C_1^2$.

Multiplying by $t_* \Phi(x_*)$ then yields
\begin{align*}
0 &\geq \frac{2(1-\theta)}{d} (\Phi G)^2 - \Phi G \Big(\frac{C_3}{\rho^2} t_* + \frac{2 C_1}{\rho} \Phi^\frac{1}{2} t_* \vert \nabla v \vert -\frac{4\theta(1-\theta)}{d} \vert \nabla v \vert^2 t_* \Phi + \Phi \Big) \\ &= \frac{2(1-\theta)}{d} (\Phi G)^2 - \Phi G \Bigg(\frac{C_3}{\rho^2} t_* - t_* \Bigg[2 \vert \nabla v \vert \sqrt{\frac{\Phi \theta (1-\theta)}{d}} - \frac{C_1}{2\rho} \sqrt{\frac{d}{\theta (1-\theta)}}\;\Bigg]^2 \\ & \qquad \qquad + \frac{d C_1^2}{4\rho^2 \theta (1-\theta)} t_*+ \Phi \Bigg) \\ &\geq  \frac{2(1-\theta)}{d} (\Phi G)^2 - \Phi G \Big(\frac{C_3}{\rho^2} t_* + \frac{d C_1^2}{4\rho^2 \theta (1-\theta)} t_* + 1 \Big),
\end{align*}
which is equivalent to
\begin{align*}
\Phi G &\leq \frac{d}{2(1-\theta)} \Big(\frac{C_3}{\rho^2} t_* + \frac{d C_1^2}{4\rho^2 \theta (1-\theta)} t_* + 1 \Big) \\ &= \frac{d}{2(1-\theta)} +\frac{d}{2(1-\theta) \rho^2} \Big(C_3 + \frac{d C_1^2}{4} \frac{1}{\theta(1-\theta)}\Big)t_* \\ &\leq \frac{d}{2(1-\theta)} +\frac{C}{2(1-\theta) \rho^2} \Big(1 + \frac{1}{\theta(1-\theta)}\Big)T
\end{align*}
with $C=\max \big\{d C_3,\frac{d^2 C_1^2}{4} \big\}$. Finally, for $(x,y) \in B_{ \rho} (z) \times Y$ we conclude that
\begin{align*}
(1-\theta) \big( \vert \nabla v \vert^2 &(T,x,y) + \Psi_\Upsilon (v) (T,x,y) \big) - \partial_t v (T,x,y) = \frac{1}{T} G(T,x,y) = \frac{1}{T} (\Phi G) (T,x,y) \\ &\leq \frac{1}{T} (\Phi G) (t_*,x_*,y_*) \leq \frac{d}{2(1-\theta)T} +\frac{C}{2(1-\theta) \rho^2} \Big(1 + \frac{1}{\theta(1-\theta)}\Big).
\end{align*}
The statement now follows as $T>0$ was chosen arbitrarily.
\end{proof}


For global solutions we obtain the following Li-Yau estimate.

\begin{korollar}
\label{KorollarLY}
Let $u: (0,\infty) \times \R^n \times Y \to (0,\infty)$ be a smooth solution to $\partial_t u - \mathcal{L} u= 0$ on $(0,\infty) \times \R^n \times Y$. Assume that there exists some $d\in [1,\infty)$ such that the operator $\mathcal{L}$ satisfies $CD_{hyb} (0,d)$. Then the function $v = \log u$ satisfies
\begin{equation}
- \mathcal{L} v = \big( \vert \nabla v \vert^2 + \Psi_\Upsilon (v)\big) - \partial_t v \leq \frac{d}{2t}\;\; \text{ on } (0,\infty) \times \R^n \times Y.
\label{LYInequGlob}
\end{equation}
\end{korollar}

\begin{proof}
This follows directly from Theorem \ref{SatzLY} by letting first $\rho \to \infty$ and then $\theta \to 0$ and applying Lemma \ref{LemmaEvolEqu}.
\end{proof}

\section{Harnack inequality}
\label{SecHarnack}

The Li-Yau inequalities from the previous section yield pointwise Harnack estimates. This can be shown by combining the integration arguments from the continuous case (see e.g. \cite[Corollary 12.3]{Li}) and the discrete case (see e.g. \cite[Theorem 6.1]{DKZ}). The basic idea is to integrate the differential Harnack inequality along paths, where 'integration' w.r.t.\ the discrete part has to be interpreted in 
an appropriate sense.  

We point out that in the following theorem the finiteness of $Y$ is not
required but the graph is assumed to be locally finite and connected, the latter meaning that
for any $y_1, y_2 \in Y$ there exists a sequence $(z_i)_{i = 0, \dots,l}$ such that $z_0 = y_1$, $z_l = x_2$ and $k(z_i,z_{i+1}) >0$ for all $i \in \{0, \dots,l-1\}$. 
In this situation, for $y_1,y_2\in Y$ by $\dist(y_1,y_2)$ we denote the length of the shortest path connecting $y_1$ and $y_2$.
Here by the length of a path we mean the total number of involved edges. 

\begin{satz}
\label{SatzHarnack}
Let $Y$ be a finite or countably infinite set and $k: Y\times Y\rightarrow [0,\infty)$ be such that the underlying graph is locally finite and connected and that we have $k(y_1,y_2)>0$ if and only if 
$k(y_2,y_1)>0$, for all $y_1,y_2\in Y$. Assume further that there exists
$k_{\min}>0$ such that $k(y_1,y_2)\ge k_{\min}$ for all $y_1,y_2\in Y$ with $k(y_1,y_2)>0$.
Let $\rho>0$ and $z \in \R^n$. Let $u$ be a positive smooth function on $(0,\infty) \times B_{\rho} (z) \times Y$ and assume that there exist constants $C,d >0$ independent
of $\rho, z, u$ and $\theta\in (0,1)$ such that $u$ satisfies the differential Harnack estimate  
\begin{equation}
(1-\theta) \big(\vert \nabla \log u \vert^2 + \Psi_\Upsilon ( \log u)\big) - \partial_t \log u \leq \frac{d}{2t(1-\theta)} + \frac{C}{2(1-\theta)\rho^2} \Big(1+\frac{1}{\theta(1-\theta)}\Big)
\label{LiYauEquHarnack}
\end{equation}
on $(0,\infty) \times B_{\rho}(z)\times Y$ for all $\theta \in (0,1)$. Then for any $x_1,x_2 \in B_{ \rho}(z)$, any $y_1,y_2\in Y$ and any $0 < t_1 < t_2 < \infty$ there holds
\begin{align}
u(t_1,x_1,y_1) \leq u(t_2,x_2,y_2) &\Big(\frac{t_2}{t_1}\Big)^\frac{d}{2(1-\theta)} \exp \Big( \frac{\vert x_2 - x_1 \vert^2}{4(1-\theta) (t_2-t_1)} + 2 \frac{\dist (y_1,y_2)^2}{(1-\theta) k_{\min} (t_2-t_1)} \notag \\ &+\frac{C (t_2-t_1)}{2(1-\theta)\rho^2} \big(1+\frac{1}{\theta (1-\theta)}\big) \Big).
\label{Harnackinequ}
\end{align} 
\end{satz}

\begin{bemerkung}
The right-hand side of \eqref{LiYauEquHarnack} coincides with the right-hand side in the Li-Yau
estimate in Theorem \ref{SatzLY}. Clearly, Theorem \ref{SatzHarnack} can also be formulated with a more general right-hand side in \eqref{LiYauEquHarnack}, resulting in a corresponding right-hand side in the Harnack inequality.
\end{bemerkung}

\begin{proof}
Let $0<t_1<t_2$, $x_1,x_2 \in B_{\rho}(z), y_1,y_2 \in Y$. Let $\gamma: [t_1,t_2] \to B_{\rho} (z), \ \gamma (t) = x_2 \frac{t-t_1}{t_2-t_1} +x_1 \frac{t_2-t}{t_2-t_1}$. Then $\gamma(t_1) = x_1$, $\gamma (t_2) = x_2$ and $\gamma^\prime (t) = \frac{x_2-x_1}{t_2-t_1}$ for all $t\in [t_1,t_2]$. For $s \in [t_1,t_2]$ we have
\begin{align*}
\log &\frac{u(t_1,x_1,y_1)}{u(t_2,x_2,y_2)} = \log \frac{u(t_1,x_1,y_1)}{u(s,\gamma(s),y_1)} + \log \frac{u(s,\gamma(s),y_1)}{u(s,\gamma(s),y_2)} +\log \frac{u(s,\gamma (s) ,y_2)}{u(t_2,x_2,y_2)} \\ &= \int_s^{t_1} \frac{d}{dt} \big( \log u(t,\gamma(t),y_1)\big) dt +\int_{t_2}^{s} \frac{d}{dt} \big( \log u(t,\gamma(t),y_2)\big) dt+ \log \frac{u(s,\gamma(s),y_1)}{u(s,\gamma(s),y_2)} \\ &= \int_s^{t_1} \Big(\partial_t \log u(t,\gamma(t),y_1) + \langle \nabla \log u(t,\gamma(t),y_1), \gamma^\prime (t)\rangle \Big) dt \\ &\qquad + \int_{t_2}^{s} \Big(\partial_t \log u(t,\gamma(t),y_2) + \langle \nabla \log u(t,\gamma(t),y_2), \gamma^\prime (t)\rangle \Big) dt + \log \frac{u(s,\gamma(s),y_1)}{u(s,\gamma(s),y_2)}. 
\end{align*}
To shorten the notation, we also write $v = \log u$ in the sequel. Using (\ref{LiYauEquHarnack}) and Cauchy-Schwarz we obtain that
\begin{align*}
&\log \frac{u(t_1,x_1,y_1)}{u(t_2,x_2,y_2)} \leq \int_{t_1}^{s} \Big[ \frac{d}{2(1-\theta)t} + \frac{C}{2(1-\theta) \rho^2} \Big(1+\frac{1}{\theta(1-\theta)}\Big) \\& 
\qquad\qquad \qquad  -(1-\theta) \Big(\vert \nabla v \vert^2 (t,\gamma(t),y_1) + \Psi_\Upsilon (v)(t,\gamma(t),y_1)\Big) \Big]dt \\ 
&\qquad\qquad\qquad+\int_{s}^{t_2} \Big[ \frac{d}{2(1-\theta)t} + \frac{C}{2(1-\theta) \rho^2} \Big(1+\frac{1}{\theta(1-\theta)}\Big) \\
 &\qquad \qquad\qquad -(1-\theta) \Big(\vert \nabla v \vert^2 (t,\gamma(t),y_2) + \Psi_\Upsilon (v)(t,\gamma(t),y_2)\Big) \Big]dt \\ &\qquad - \int_{t_1}^s \langle \nabla v(t,\gamma(t),y_1), \gamma^\prime (t)\rangle dt - \int_{s}^{t_2} \langle \nabla v(t,\gamma(t),y_2), \gamma^\prime (t)\rangle dt + \log \frac{u(s,\gamma(s),y_1)}{u(s,\gamma(s),y_2)} \\ &\leq \int_{t_1}^{t_2} \frac{d}{2(1-\theta)t} dt + \frac{C (t_2-t_1)}{2(1-\theta) \rho^2} \Big(1+\frac{1}{\theta(1-\theta)}\Big) + \log \frac{u(s,\gamma(s),y_1)}{u(s,\gamma(s),y_2)} \\ & \qquad + \int_{t_1}^s \Big(\vert \nabla v \vert (t,\gamma(t),y_1) \vert \gamma^\prime (t) \vert - (1-\theta) \vert \nabla v \vert^2 (t,\gamma(t),y_1) \Big) dt\\
& \qquad -(1-\theta) \int_{t_1}^s \Psi_\Upsilon (v)(t,\gamma(t),y_1) dt \\ & \qquad + \int_{s}^{t_2} \Big( \vert \nabla v \vert (t,\gamma(t),y_2) \vert \gamma^\prime (t) \vert - (1-\theta) \vert \nabla v \vert^2 (t,\gamma(t),y_2) \Big) dt\\
&\qquad -(1-\theta) \int_{s}^{t_2} \Psi_\Upsilon (v)(t,\gamma(t),y_2) dt.
\end{align*}

Now, for any $t\in [t_1,t_2]$ and any $(x,y) \in \R^n \times Y$, denoting $\vert \nabla v \vert = \vert \nabla v (t,x,y) \vert$, we have
\begin{align*}
\vert \nabla v \vert \vert \gamma^\prime (t) \vert - (1-\theta) \vert \nabla v \vert^2 &= -\Big( \sqrt{1-\theta} \vert \nabla v \vert - \frac{1}{2 \sqrt{1-\theta}} \vert \gamma^\prime (t) \vert\Big)^2 + \frac{1}{4(1-\theta)} \vert \gamma^\prime (t) \vert^2 \\ &\leq \frac{1}{4(1-\theta)} \vert \gamma^\prime (t) \vert^2
\end{align*}
and thus
\begin{align*}
&\int_{t_1}^s  \Big(\vert \nabla v \vert (t,\gamma(t),y_1) \vert \gamma^\prime (t) \vert - (1-\theta) \vert \nabla v \vert^2 (t,\gamma(t),y_1) \Big) dt \\ &  + \int_{s}^{t_2} \Big( \vert \nabla v \vert (t,\gamma(t),y_2) \vert \gamma^\prime (t) \vert - (1-\theta) \vert \nabla v \vert^2 (t,\gamma(t),y_2) \Big) dt \leq \int_{t_1}^{t_2} \frac{1}{4(1-\theta)} \vert \gamma^\prime (t) \vert^2 dt.
\end{align*}

Turning to the remaining terms, note first that in the case $y_1=y_2$ we trivially have 
\begin{align*}
\log &\frac{u(s,\gamma(s),y_1)}{u(s,\gamma(s),y_2)} -(1-\theta) \int_{t_1}^s \Psi_\Upsilon (v)(t,\gamma(t),y_1) dt -(1-\theta) \int_{s}^{t_2} \Psi_\Upsilon (v)(t,\gamma(t),y_2) dt \leq 0,
\end{align*}
and thus the assertion follows by calculating the integrals in the same way as we do below.

Consider next the case in which $k(y_1,y_2) >0$, that is $y_1$ and $y_2$ are neighbours. We now follow the line of arguments from the proof of \cite[Theorem 6.1]{DKZ} and estimate
\begin{align*}
\log &\frac{u(s,\gamma(s),y_1)}{u(s,\gamma(s),y_2)} -(1-\theta) \int_{t_1}^s \Psi_\Upsilon (v)(t,\gamma(t),y_1) dt -(1-\theta) \int_{s}^{t_2} \Psi_\Upsilon (v)(t,\gamma(t),y_2) dt \\ &\leq \log \frac{u(s,\gamma(s),y_1)}{u(s,\gamma(s),y_2)} -(1-\theta) k(y_2,y_1) \int_{s}^{t_2} \Upsilon \big(v(t,\gamma(t),y_1)-v(t,\gamma(t),y_2)\big) dt \\ &\leq \delta(s) - (1-\theta) k_{\min} \int_s^{t_2} \Upsilon \big( \delta (t)\big) dt
\end{align*}
with $\displaystyle \delta(t) = \log \frac{u(t,\gamma(t),y_1)}{u(t,\gamma(t),y_2)}$. Choosing $s \in [t_1,t_2]$ in such a way that the continuous function $\omega$ defined by
\begin{align*}
\omega (t) = \delta (t) - (1-\theta) k_{\min} \int_t^{t_2} \Upsilon \big( \delta (\tau)\big) d\tau,\quad t\in [t_1,t_2],
\end{align*}
attains its minimum at $s$ and arguing as in the proof of \cite[Theorem 6.1]{DKZ} we obtain that 
\begin{align*}
\omega (s) \leq \frac{2}{(1-\theta)k_{\min} (t_2-t_1)}.
\end{align*}

Summarizing the previous estimates, we see that
\begin{align*}
\log \frac{u(t_1,x_1,y_1)}{u(t_2,x_2,y_2)}  &\leq \int_{t_1}^{t_2} \frac{d}{2(1-\theta)t} dt + \frac{C (t_2-t_1)}{2(1-\theta) \rho^2} \big(1+\frac{1}{\theta(1-\theta)}\big) \\ &\qquad  + \int_{t_1}^{t_2} \frac{1}{4(1-\theta)} \vert \gamma^\prime (t) \vert^2 dt + \frac{2}{(1-\theta)k_{\min} (t_2-t_1)}.
\end{align*}

Let now $y_1,y_2 \in Y$ be arbitrary with $l := \dist (y_1,y_2)>1$ and let $(z_i)_{i=0,1,\dots,l}$ be a minimal sequence connecting $y_1$ and $y_2$, i.e. $z_0=y_1$, $z_l = y_2$ and $k(z_i,z_{i+1}) >0$ for all $i \in \{0,\dots,l-1\}$. Let further $(\tau_i)_{i=0,\dots,l}$ be defined by $\tau_i = t_1+i \frac{t_2-t_1}{l}$. Then, as all points $\big(\tau_i, \gamma (\tau_i)\big)$ lie on the straight line connecting $(t_1,x_1)$ and $(t_2,x_2)$, we may infer that
\begin{align*}
\log \frac{u(t_1,x_1,y_1)}{u(t_2,x_2,y_2)} &= \sum_{i=0}^{l-1} \log \frac{u(\tau_i,\gamma(\tau_i),z_i)}{u(\tau_{i+1},\gamma (\tau_{i+1}),z_{i+1})} \\ &\leq \int_{t_1}^{t_2} \frac{d}{2(1-\theta)t} dt + \sum_{i=0}^{l-1} \frac{C (\tau_{i+1}-\tau_i)}{2(1-\theta) \rho^2} \big(1+\frac{1}{\theta(1-\theta)}\big) \\ &\qquad  + \int_{t_1}^{t_2} \frac{1}{4(1-\theta)} \vert \gamma^\prime (t) \vert^2 dt + \sum_{i=0}^{l-1} \frac{2}{(1-\theta)k_{\min} (\tau_{i+1}-\tau_i)} \\ &= \log \big(\frac{t_2}{t_1}\big)^{\frac{d}{2(1-\theta)}} + \frac{C (t_2-t_1)}{2(1-\theta) \rho^2} \big(1+\frac{1}{\theta(1-\theta)}\big) \\ &\qquad + \frac{\vert x_2-x_1 \vert^2}{4(1-\theta) (t_2-t_1)} + \frac{2 l^2}{(1-\theta) k_{\min} (t_2-t_1)},
\end{align*}
from which (\ref{Harnackinequ}) follows by applying the exponential function on both sides.
\end{proof}

\begin{korollar}
\label{KorollarHarnack}
Let $Y, k$ and $k_{\min}$ be as in Theorem \ref{SatzHarnack} and $u: (0,\infty) \times \R^n \times Y \to (0,\infty)$ be a smooth solution to the equation $\partial_t u - \mathcal{L} u= 0$ on $(0,\infty) \times \R^n \times Y$. Suppose that $\mathcal{L}$ satisfies $CD_{hyb} (0,d)$ with some $d\in [1,\infty)$. Then for any $x_1,x_2 \in \R^n$,
any $y_1,y_2\in Y$ and any $0 < t_1 < t_2 < \infty$ there holds
\begin{align}
u(t_1,x_1,y_1) \leq u(t_2,x_2,y_2) \Big(\frac{t_2}{t_1}\Big)^\frac{d}{2} \exp \Big( \frac{\vert x_2 - x_1 \vert^2}{4 (t_2-t_1)} + 2 \frac{\dist (y_1,y_2)^2}{k_{\min} (t_2-t_1)}\Big).
\label{HarnackGlob}
\end{align} 
\end{korollar}

\begin{proof}
This follows from Theorem \ref{SatzHarnack} by letting first $\rho \to \infty$ and then $\theta \to 0$ having in mind that Theorem \ref{SatzLY} ensures the validity of (\ref{LiYauEquHarnack}) as $\mathcal{L}$ satisfies $CD_{hyb}(0,d)$.
\end{proof}

\begin{bemerkung}
Our Harnack inequality from (\ref{HarnackGlob}) is closely related to both, the continuous and the discrete one for solutions to the respective version of the heat equation. To be more concrete, when $y_1=y_2$, equation (\ref{HarnackGlob}) gives the Harnack inequality from \cite[Corollary 12.5]{Li} (with increased dimension exponent $d$ instead of $n$) and in the case $x_1=x_2$ we obtain a version of \cite[Theorem 6.1]{DKZ} in the purely discrete case, again with increased dimension parameter. 
%
\end{bemerkung}

\begin{beispiel} \label{Beispiele}
(i) \emph{2-point graph.} We come back to Example \ref{2-pkteGraph}, where the underlying
discrete space is a 2-point graph with symmetric weight $\eta>0$. In the case $\eta=1$, we know
from Example \ref{BspVollstGraph} that the operator $\mathcal{L} = \Delta \oplus L_d$ satisfies $CD_{hyb} (0,d)$ with $d = n+\frac{2}{C(2)} = n+\frac{1}{C_{\log}} (\approx n+1,258$, see Remark \ref{BemerkungMuench}). This result extends to the general case $\eta>0$, since
a constant weight $\eta$ leads to an additional $\eta^2$-term in both the $\Psi_{2,\Upsilon}(u)$
and the $(L_d(u))^2$ term and thus cancels out, cf.\ the calculation in
Example \ref{2-pkteGraph}. Corollary \ref{KorollarHarnack} now shows that for any smooth 
pair of positive functions $(u_1,u_2)$ which solves \eqref{firstexample} on $(0,\infty)\times \iR^n$, there holds
\begin{align*}
u_i(t_1,x_1) \leq u_j(t_2,x_2) \big(\frac{t_2}{t_1}\big)^{\frac{n}{2}+\frac{1}{2C_{\log}}} \exp \Big(\frac{\vert x_2-x_1 \vert^2}{4(t_2-t_1)}+\frac{2 \dist(y_i,y_j)^2}{\eta(t_2-t_1)}\Big), \ i,j \in \{1,2\},
\end{align*}
for any $0<t_1<t_2<\infty$ and $x_1,x_2 \in \R^n$.
This Harnack inequality includes two single-species estimates; indeed, for $j=i$ we find 
\begin{align*}
u_i (t_1,x_1) \leq u_i (t_2,x_2) \big(\frac{t_2}{t_1}\big)^{\frac{n}{2}+\frac{1}{2C_{\log}}} \exp \Big(\frac{\vert x_2-x_1 \vert^2}{4(t_2-t_1)}\Big), \ i=1,2,\ 0<t_1<t_2,\ x_1,x_2\in \iR^n.
\end{align*}
Setting $x_1=x_2=x$ we further obtain the interesting estimate
\begin{align*}
u_i (t_1,x) \leq u_j (t_2,x) \big(\frac{t_2}{t_1}\big)^{\frac{n}{2}+\frac{1}{2C_{\log}}} \exp \Big(\frac{2}{\eta(t_2-t_1)}\Big), \ i \neq j,\ 0<t_1<t_2,\ x\in \iR^n.
\end{align*}

(ii) \emph{The square.} Let $u_i$, $i \in \{1,2,3,4\}$, be smooth positive functions which solve the system
\begin{equation} \label{squareExam}
\left\{
\begin{array}{r@{\;=\;}l@{\;}l}
\partial_t u_1 (t,x) - \Delta u_1 (t,x) & u_2(t,x)+u_4(t,x) -2 u_1 (t,x),\quad & t>0,\,x\in \iR^n, \\ 
\partial_t u_2 (t,x) - \Delta u_2 (t,x) & u_1(t,x)+u_3(t,x) -2u_2 (t,x),\quad & t>0,\,x\in \iR^n, \\ 
\partial_t u_3 (t,x) - \Delta u_3 (t,x) & u_2(t,x)+u_4(t,x) -2u_3 (t,x),\quad & t>0,\,x\in \iR^n, \\ 
 \partial_t u_4 (t,x) - \Delta u_4 (t,x) & u_1(t,x)+u_3(t,x) -2u_4 (t,x),\quad & t>0,\,x\in \iR^n.
\end{array}
\right.
\end{equation}
Letting $Y= \{y_1,\dots,y_4\}$ be a four-element set and setting $u(t,x,y_i)=u_i(t,x)$ for $i=1,\dots,4$, the system \eqref{squareExam} is equivalent to 
\begin{align*}
\partial_t u(t,x,y)-\Delta u(t,x,y) -L_d (t,x,y)=0,\quad t\in (0,\infty),\ x\in \iR^n,\ y\in Y,
\end{align*}
where $L_d$ is generated by the kernel $k$ with $k(y_i,y_j) = 1$ if $\vert i-j \vert \in \{1,3\}$ and $k=0$ else. The underlying graph structure is the square that we already mentioned in Example \ref{BeispielRicci} (ii). There we showed that the corresponding operator $\mathcal{L}=\Delta \oplus L_d$ satisfies $CD_{hyb}(0,d)$ with $d=n+\frac{4}{C(2)} = n+\frac{2}{C_{\log}} (\approx n+2,516$, see again Remark \ref{BemerkungMuench}). Thus, Corollary \ref{KorollarHarnack} yields the Harnack inequality
\begin{align*}
u_i(t_1,x_1) \leq u_j (t_2,x_2) \Big(\frac{t_2}{t_1}\Big)^{\frac{n}{2}+\frac{1}{C_{\log}}} \exp \Big( \frac{\vert x_2 - x_1 \vert^2}{4 (t_2-t_1)} + \frac{\dist(y_i,y_j)^2}{ (t_2-t_1)}\Big) 
\end{align*}
for any $0<t_1<t_2<\infty$, $x_1,x_2 \in \R^n$ and $i,j\in \{1,2,3,4\}$.
\end{beispiel}

\section{Open problems}
\label{Section5}
In this section we want to describe several open problems with regard to possible extensions
of our results.

{\em 1. Infinite graphs.} Our results on Li-Yau estimates, Theorem \ref{SatzLY} and Corollary \ref{KorollarLY}, are restricted to finite sets $Y$. To establish similar results for infinite but locally finite graphs seems to be possible but is
more difficult. The problem is that one also needs a cutoff function in the discrete variable. For
solutions which are global w.r.t.\ $y$, one may possibly argue as in \cite[Section 5]{KWZ}.
For solutions which are local w.r.t.\ $y$, the situation is much more challenging due to delicate
commutator estimates, see \cite{Harvard} and \cite{DKZ} for corresponding results in the purely
discrete setting. Another difficulty arises if the kernel $k$ is such that arbitrary long jumps may 
occur meaning that the graph is not locally finite. In \cite{KWZ,SWZ1} it was shown that for certain operators of the form
\begin{equation} \label{genlaplacedef}
L_d v(y)=\,\sum_{j\in \iZ} \tilde{k}(j) \big(v(y+j)-v(y)\big),\quad y\in \iZ,
\end{equation}
with symmetric, nonnegative and summable kernel $\tilde{k}$ without finite second moment (including the fractional 
discrete Laplacian), the condition $CD_\Upsilon(0,d)$ fails to hold for any $d\in [1,\infty)$, whereas a corresponding condition with 
a nonquadratic dimension term (i.e.\ a nonquadratic $CD$-function) is satisfied. To make use of
the latter would require a significant more general version of the hybrid tensorisation principle.

{\em 2. Different diffusion coefficients.} In \eqref{RDSystems}, the diffusion coefficient of each species is equal to one. Clearly, by scaling in time, we could allow also
for a different positive diffusion coefficient, say $d$, but we are not able to treat the case with different diffusion coefficients. 
So, the question is whether it is possible to prove a Li-Yau inequality for more general systems of the form
\begin{equation} \label{genRDSystems}
\partial_t u_i (t,x) - d_i\Delta u_i(t,x) = \sum_{j=1}^m k(i,j)\big(u_j(t,x)-u_i(t,x)\big), 
\quad t>0,\ x\in \iR^n,\ i = 1,\dots,m.
\end{equation}
The main additional difficulty is that the two operators in the sum 
\[
\mathcal{L}u(x,y)=d(y)\Delta u(x,y)+L_d u(x,y)\]
 no longer
commute, due to the variable coefficient $d(y)$. As a consequence, the crucial inequality from Lemma \ref{LemmaCrucInequ} breaks down and thus also the hybrid tensorisation principle is no longer available.

{\em 3. Nonlocal diffusion.} If the Laplace operator $\Delta$ is replaced with a fractional
Laplacian $-(-\Delta)^\beta$ with $\beta\in (0,1)$, then one faces the problem that
no suitable curvature-dimension condition is known for the continuous part. In fact, it was shown
in \cite{SWZ2} that the fractional Laplacian on $\iR^n$ fails to satisfy the condition $CD(0,n)$ for all $n\in [1,\infty)$. Thus our approach based on $CD$-conditions does not seem to be feasible 
in the nonlocal case. On the other hand, it was recently shown in \cite{WZ23} that positive global solutions of the space fractional heat equation do fulfil a Li-Yau inequality, from which one can
also deduce a Harnack inequality. So, there is hope that Li-Yau estimates also hold true 
for \eqref{RDSystems} with $\Delta$ replaced by $-(-\Delta)^\beta$ with $\beta\in (0,1)$.


$\mbox{}$
{\footnotesize

$\mbox{}$


}


\begin{thebibliography}{99}
{\footnotesize
\bibitem{AUB}
G. Auchmuty, D. Bao: Harnack-type inequalities for evolution equations. Proceedings of the American Mathematical Society, {\bf 122} (1994), 117--129. 
\bibitem{BGL} 
D. Bakry, I. Gentil, M. Ledoux: \textit{Analysis and geometry of Markov diffusion operators}. Grundlehren der mathematischen Wissenschaften \textbf{348}. Springer Cham, Heidelberg, 2014.
\bibitem{BE}
D. Bakry, M. Émery: \textit{Diffusions hypercontractives}. Séminaire de probabilités, XIX, 1983/84, 177-206. Lecture Notes in Math., Springer, Berlin, 1985.
\bibitem{BL}
D. Bakry, M. Ledoux: A logarithmic Sobolev form of the Li-Yau parabolic inequality. Rev. Mat. Iberoam. {\bf 22} (2006),
683--702.
\bibitem{Harvard} 
F. Bauer, P. Horn, Y. Lin, G. Lippner, D. Mangoubi, S.-T. Yau: Li-Yau inequality on graphs. J. Differential Geom. \textbf{99} (2015), 359--405.
\bibitem{CY}
F.R.K. Chung, S.T. Yau: Logarithmic Harnack inequalities, Math. Res. Lett. 3 (6) (1996), 793--812.
\bibitem{CKK}
D. Cushing, S. Kamtue, R. Kangaslampi, S. Liu, N. Peyerimhoff: Curvature, Graph Products and Ricci flatness, J. Graph Theory 96 (2021), no. 4, 522--553.
\bibitem{DKZ} 
D. Dier, M. Kassmann, R. Zacher: Discrete versions of the Li-Yau gradient estimate. 
Ann.\ Sc.\ Norm.\ Super.\ Pisa Cl.\ Sci.\ (5) {\bf 22} (2021), 691--744.
\bibitem{EM12}
M.\ Erbar, J.\ Maas: Ricci curvature of finite Markov chains via convexity of the entropy. Arch.\ Ration.\ Mech.\ Anal.\ {\bf 206} (2012), 997--1038.
\bibitem{FoePol}
J.\ F\"oldes, P.\ Pol\'a\^{c}ik:
On cooperative parabolic systems: Harnack inequalities and asymptotic symmetry.
Discrete Contin. Dyn. Syst. {\bf 25} (2009), 133--157. 
\bibitem{KWZ} 
S. Kräss, F. Weber, R. Zacher: Li-Yau and Harnack inequalities via curvature-dimension conditions for discrete long-range jump operators including the fractional discrete Laplacian. Preprint 2022. \url{https://arxiv.org/abs/2201.04564}.
\bibitem{Li}
P. Li: \textit{Geometric analysis}. Cambridge Studies in Advanced Mathematics. Cambridge University Press, Cambridge, 2012.
\bibitem{LY} 
P. Li,\ S.-T. Yau: On the parabolic kernel of the Schr\"odinger operator. Acta Math. {\bf 156} (1986), 153--201.
\bibitem{Mie}
A.\ Mielke: Geodesic convexity of the relative entropy in reservible Markov chains. Calc.\ Var.\
Partial Diff.\ Equ.\ {\bf 48} (2013), 1--31.
\bibitem{Mon}
P.\ Monmarch\'e: Generalized $\Gamma$ calculus and application to interacting particles on a graph. Potential Anal. {\bf 50} (2019), no. 3, 439--466.
\bibitem{Moser64} J.\ Moser: A Harnack inequality for parabolic
differential equations. Comm.\ Pure Appl.\ Math. {\bf 17} (1964),
101--134.
\bibitem{MUN} 
F. M\"unch: Li-Yau inequality on finite graphs via non-linear curvature dimension conditions. J. Math. Pures Appl. (9) {\bf 120} (2018), 130--164. 
\bibitem{SWZ1}
A.\ Spener, F.\ Weber, R.\ Zacher: Curvature-dimension inequalities for non-local operators in the discrete setting. Calc.\ Var.\ Partial Differential Equations {\bf 58} (2019), Paper No. 171,
\bibitem{SWZ2}
 Spener, A.; Weber, F.; Zacher, R.: The fractional Laplacian has infinite dimension. Comm.\ Partial Differential Equations 
 {\bf 45} (2020), 57--75.
\bibitem{We}
F.\ Weber: Entropy-information inequalities under curvature-dimension conditions for continuous-time Markov chains. Electron. J. Probab. 26 (2021), Paper No. 52
\bibitem{WZ23}
F.\ Weber, R.\ Zacher: Li-Yau inequalities for general non-local diffusion equations via reduction to the heat kernel. Math.\ Ann.\ {\bf 385} (2023), 393--419.
\bibitem{WZ} 
F.\ Weber, R. Zacher: The entropy method under curvature-dimension conditions in the spirit of Bakry-\'Emery in the discrete setting of Markov chains.
 J.\ Funct.\ Anal.\ 281 (2021), no. 5, Paper No. 109061.
}
\end{thebibliography}
\end{document}